\documentclass[10pt,a4paper]{amsart}
\usepackage[latin2]{inputenc}
\usepackage[T1]{fontenc}

\usepackage{euscript}
\usepackage{eufrak}
\usepackage{amsmath}
\usepackage{amsthm}
\usepackage{amssymb}

\usepackage{enumerate}

\usepackage[bookmarksnumbered,bookmarksopen,pdfusetitle,unicode]{hyperref}

\numberwithin{equation}{section}
\numberwithin{equation}{subsection}

\theoremstyle{plain}
\newtheorem{theorem}[equation]{Theorem}
\newtheorem{lemma}[equation]{Lemma}
\newtheorem{proposition}[equation]{Proposition}
\newtheorem{corollary}[equation]{Corollary}
\newtheorem{definition}[equation]{Definition}

\newtheorem{deflem}[equation]{Definition/Lemma}

\theoremstyle{definition}
\newtheorem{example}[equation]{Example}
\newtheorem{remark}[equation]{Remark}

\DeclareMathOperator{\Hom}{Hom}

\newcommand{\cS}{\mathcal{S}}

\newcommand{\cO}{\mathcal{O}}
\newcommand{\cV}{\mathcal{V}}
\newcommand{\cF}{\mathcal{F}}
\newcommand{\cU}{\mathcal{U}}
\newcommand{\cE}{\mathcal{E}}
\newcommand{\cI}{\mathcal{I}}

\newcommand{\cG}{\mathcal{G}}

\newcommand{\bt}{{\bf t}}
\newcommand{\bn}{{\bf n}}
\newcommand{\bx}{{\bf x}}

\newcommand{\setC}{\mathbb{C}}
\newcommand{\setQ}{\mathbb{Q}}

\newcommand{\setZ}{\mathbb{Z}}
\newcommand{\setN}{\mathbb{N}}
\newcommand{\setP}{\mathbb{P}}

\newcommand{\bI}{{J(l',I)}}
\newcommand{\ring}{\setZ[[\bt^{\pm 1}]]}

\newcommand{\ringd}{\setZ[[\bt^{\pm 1/d}]]}
\newcommand{\hh}{\mathfrak{h}}
\newcommand{\pp}{\mathfrak{p}}

\providecommand{\coloneqq}{\mathrel{:=}}

\begin{document}

\title[Poincar\'e series of surface singularities]{Poincar\'e series associated
with surface singularities}


\author[A.~Némethi]{András Némethi}
\thanks{The author is partially supported by OTKA grants}
\address{Rényi Institute of Mathematics,  1053 Budapest,
Reáltanoda u. 13--15,  Hungary}
\email{nemethi@renyi.hu}



\maketitle

\begin{center}
{\em Dedicated to L\^e D\~ung Tr\'ang on his 60th birthday}
\end{center}

\begin{abstract} We unify and generalize formulas obtained by
Campillo, Delgado and Gusein-Zade in their series of articles (see
e.g. \cite{CDG,CDGc,CDGEq}). Positive results are established for
rational and minimally elliptic singularities.  By examples and
counterexamples we also try to find the `limits' of these
identities. Connections with the Seiberg-Witten Invariant Conjecture
and Semigroup Density Conjecture are discussed. 
\end{abstract}

\section{Introduction}\label{Sec:1}

In a long series of articles, Campillo, Delgado and Gusein-Zade
established an identity of type $Z=P$, see e.g.
\cite{CDG,CDGc,CDGEq} (the interested  reader is invited to
consult all their other articles in this subject as well). The identity, in
different articles had different versions and targeted slightly
different objects. One of the versions was valid for rational
normal surface singularities $(X,o)$ and for one of its fixed
resolutions $\pi$; $Z$ was a topological invariant expressed from
the combinatorics of the dual graph, while $P$ was a multi-variable
Poincar\'e series associated with a filtration given by the
divisorial valuations of the irreducible exceptional divisors. In
the echivariant setting, an action of a finite abelian group was
present too. Finally, in a relative version, they considered an
embedded curve $C$ into the rational singularity $(X,o)$, $Z$
was an Alexander type topological invariant expressed from an
embedded resolution graph, while $P$  the Poincar\'e series of a
filtration associated with the evaluations provided by the
normalization of $C$.

The authors provide independent proofs for the different cases;
the proofs are based on the structure of arc-spaces and  an
`euler-characteristic summation' (suggested by the motivic
integration).

The present article has several goals.

First, we create an uniform language to incorporate all the above
different versions (the absolute, echivariant and relative) in
only one setting. Moreover, we show that the statement of the
absolute version (the `Main Identity' \ref{MTH}) stated for all
possible good resolutions implies all the other versions. All the
statements are formulated for any normal surface singularity (with
rational homology sphere link); the formulation involves heavily
the universal abelian cover of $(X,o)$. Then we establish the
identities for some families of singularities (see below).  In the
presentation we embed the statement and the proof into the
classical singularity theory: the presentation is in the spirit of
the work of Artin and Laufer. In fact, one of the crucial
arguments is the vanishing theorem of R\"ohr (a Laufer, or
Grauert-Riemenschneider type vanishing). In this way, the present
proof not only recovers the result of Campillo, Delgado and
Gusein-Zade for rational singularities, but automatically extends
it to the minimally elliptic singularities as well. The article
also establishes a new realization of the topological invariant
$Z$, expressed in terms of (analytic) euler-characteristics of
cycles supported by the exceptional locus, which spans as a bridge
the topological and analytic invariants (see (\ref{PROPTOP})). Its
formula might have an independent interest as well.

In fact, the new unifying language also enables us to formulate
the relative version also in a much higher generality (compared
with \cite{CDGc}).

We would like to stress, that although in the case of rational
singularities it is unimportant which resolution we choose, for
all other singularities this is crucial. For minimally elliptic
singularities we establish our positive results for the minimal
(good) resolution, and by examples we show pathologies valid for
non-minimal ones.

The machinery also allows us to reduce the number of variables.
Surprisingly, if we simplify (rather much) the identities, e.g. we
consider the reduced identities with one variable, even by these
simplified versions we recover
famous classical/older (in general, hardly non-trivial) results
already present in the literature. Our wish is to present 
some of them (regarding e.g. weighted homogeneous,
superisolated or splice-quotient singularities). The relation with
the Seiberg-Witten conjecture \cite{nemethi02:_seiber_witten} and
the Semigroup Density conjecture (regarding the existence of
unicuspidal rational projective plane curves) \cite{MR2192386} is
striking (cf. (\ref{si}). (In fact, 
this connection was the motivation of the author in the subject.)

Finally, by examples, we try to establish the `limits' of the
identity $Z=P$. We provide counterexamples and explicit methods
which provide examples with $Z\not=P$, showing serious geometric
reasons which obstruct the identity,  in general.

\section{Notations and preliminaries}\label{Sec:2}

\subsection{The resolution}\label{DRG}
 Let $(X,o)$ be an isolated
complex analytic normal surface singularity.
 Let $\pi\colon\widetilde{X}\to X$
be a {\em good} resolution with exceptional set $E$ with
irreducible components $E_j$, $j\in \cV:=\{1,\ldots, s\}$. Let
$\Gamma$ denote its dual resolution graph  (for details see e.g.\
\cite{INV}). For any $j\in\cV$ we write $\delta_j$ for the valency 
of $j$ in $\Gamma$.

We assume that the link $\Sigma$ of $X$ is a rational homology
sphere, i.e. $\Gamma$ is a connected tree and  $E_j\approx\setP^1$
for every $j$.

Set $L \coloneqq H_2(\widetilde{X},\setZ)$ and \(L' \coloneqq
H_2(\widetilde{X},\Sigma,\setZ)\).  These groups are free with
bases the classes of \(E_i\) and their duals \(E^*_i\). Here, we
prefer the sign convention $(-E^*_i, E_j)=\delta_{ij}$ (the
Kronecker delta function). The matrix \(I\) of the inclusion
$L\hookrightarrow L'$ in the basis $\{E_i\}_i$ of \(L\) and
$\{-E^*_i\}_i$ of \(L'\) is exactly the intersection matrix
$(E_i,E_j)_{i,j}$.

We denote the group   $L'/L \cong H_1(\Sigma,\setZ)$
by $H$ and set $[l']$ for the class of $l'\in L'$.
Let $\lvert H \rvert $ and $\widehat{H}$ denote its order and
Pontrjagin dual $\Hom(H,\setC^*)$, respectively.
There is a natural isomorphism  $\theta:H\to \widehat{H}$ given by
$[l']\mapsto e^{2 \pi i(l',\cdot)}$.
Sometimes we also write $d\coloneqq |H|$.


\subsection{Positive  cones and representatives}\label{cones}
We write $L_\setQ$ for the group of rational cycles $L\otimes \setQ$.
The form $(\cdot,\cdot )$ has a natural extension to $L_\setQ$.

A rational cycle $x=\sum_jr_jE_j\in L_\setQ$ is called {\em
effective}, denoted by $x\geq 0$, if $r_j\geq 0$ for all $j$.
Their cone is denoted by $L_{\setQ,e}$, while
$L'_e:=L_{\setQ,e}\cap L'$ and $L_e:=  L_{\setQ,e}\cap L$.

A rational cycle $x\in L_\setQ$ is called {\em numerical
effective} if $(x,E_j)\geq 0$ for all $j$. Their cone is denoted
by $NE_\setQ$. We also write (for Lipman's cones)
$\cS_\setQ:=-NE_\setQ$, $\cS':=\cS_{\setQ}\cap L'$
 and $\cS:=\cS_\setQ\cap L$. Notice that both $\cS'$ and $\cS$ are semigroups,
$\cS'$ is generated over $\setN$ by the vectors $E^*_j$. One can
show (using the fact that $I$ is negative definite) that $E^*_j\in
L'_e$ for all $j$.

Define the `unit cube' in $L'$ by
$$Q:=\{r\in L'\,:\, \textstyle{r=\sum_jr_jE_j} \ \mbox{with $r_j\in [0,1)$}\}.$$
For any $h\in H$, we denote the unique representative of $H$ from
$Q$ by  $r_h$. 

\subsection{The ring $\setZ[[\bt^{L'}]]$}\label{ring}
Consider the notations of (\ref{DRG}).
We denote by $\ring$ the ring of formal power series
$\setZ[[t_1^{\pm 1},\ldots,t_s^{\pm 1}]]$ (where $s=|\cV|$).
In fact, it is convenient to consider a larger ring as well,
the ring $\ringd=\setZ[[t_1^{\pm 1/d},\ldots,t_s^{\pm 1/d}]]$
 of formal power series in variables $t_j^{\pm 1/d}$.
$\ringd$ has a natural sub-ring associated with the resolution
$\pi$. Namely,  $\setZ[[\bt^{L'}]]$  is the $\setZ$-linear
combinations of monomials of type
$$\bt^{l'}=t_1^{l'_1}\cdots t_s^{l'_s}, \ \ \ \mbox{where} \ \
l'=\textstyle{\sum_j\,l'_jE_j}\in L'.$$ This admits several
sub-rings, e.g. $\setZ[[\bt^{L'_e}]]$, or $\setZ[[\bt^{\cS'}]]$,
generated by monomials $\bt^{l'}$ with $l'\in L'_e$, or $l'\in
\cS'$ respectively.

In fact, $\setZ[[\bt^{\cS'}]]$ is a usual formal power series ring
in variables $\bt^{E_j^*}$. More precisely, any of its element has
the form
\begin{equation}\label{varx}
\Phi(f)(\bt):= f(\bt^{E^*_1},\ldots, \bt^{E^*_s}), \ \
\mbox{where}\ \ f(x_1,\ldots, x_s)\in\setZ[[{\bf x}]].
\end{equation}

\begin{definition}\label{hcompa}
Any series $S(\bt)=\sum_{l'}a_{l'}\bt^{l'}\in \setZ[[\bt^{L'}]]$
decomposes in a unique way as \begin{equation}\label{hdec}
S=\sum_{h\in H}S_h,\ \ \mbox{where} \ \
S_h=\sum_{[l']=h}a_{l'}\bt^{l'}.\end{equation} $S_h$ is called the
$h$-component of $S$.
\end{definition}

\begin{lemma}\label{hcompb} Consider $F(\bt):=\Phi(f)(\bt)$ for some
$f\in \setZ[[{\bf x}]]$ as in (\ref{varx}). Then
$$F_h(\bt)=
 \frac{1}{\lvert H \rvert} \cdot \sum_{\rho
\in \widehat{H}}\, \rho(h)^{-1}\cdot f(
\rho([E^*_1])\bt^{E^*_1},\ldots, \rho([E^*_s]) \bt^{E^*_s}).$$
\end{lemma}
\begin{proof}
First notice that $\sum_hF_h=F$, hence it is enough to show that
$F_h$ is an $h$-component. Indeed, if $\prod x_j^{n_j}$ is a
monomial of $f$, and $l':=\sum n_jE^*_j$, then $(1/d)\sum_\rho
\rho(-h)\rho([l'])\bt^{l'}$ is $\bt^{l'}$ if $[l']=h$ and zero
otherwise.
\end{proof}

\section{The topological/combinatorial invariant}

\subsection{The rational functions
$\mathbf{ \zeta(\bt)}$}\label{zeta} Our main topological object is
the rational function (in variables $x_j=\bt^{E_j^*}$), or its Taylor
expansion at the origin (cf. \ref{varx}):
\begin{equation}\label{zetaelso}
\zeta (\bt):=\Phi(z)(\bt), \ \ \mbox{where} \ \ z({\bf x}):=
\prod_{j\in \cV}(1-x_j)^{\delta_j-2}.
\end{equation}
Notice that by (\ref{hcompb}), its $h$-component is
\begin{equation}\label{zeta2}
\zeta_h(\bt) \coloneqq \frac{1}{\lvert H \rvert} \cdot \sum_{\rho
\in \widehat{H}}\, \rho(h)^{-1}\cdot \prod_{j\in\cV}
{(1-\rho([E^*_j]) \bt^{E^*_j})}^{\delta_j-2}.
\end{equation} $\zeta_h$ is the `multivariable' version of the
rational function used by
\cite{nemethi02:_seiber_witten,nemethi04:_seiber_witten,nemethi05:_seiber_witten}
in the Reidemeister-Turaev computation of $\Sigma$, by \cite{Opg}
in the computation of the geometric genus of splice-quotient
singularities, by \cite{NO1,NO2} in the proof of the Neumann-Wahl
Casson Invariant Conjecture for splice-quotients; see also
\cite{BN}.  The multivariable version appears in \cite{CDGEq}. See
also the comments of (\ref{toppart}).

\subsection{The CDG-series} In their series of articles, 
 Campillo, Delgado and Gusein-Zade  (see
e.g. \cite{CDG,CDGEq}) considered the next infinite series. In fact, it 
tautologically equals the Taylor expansion at the
origin of $\zeta$. One starts with the following coefficients: for
any $n\geq 0$ and $\delta\geq 0$, set $\chi_{\delta,n}$ as the
coefficient of $x^n$ in the Taylor expansion at the origin of the
the function $(1-x)^{\delta-2}$. Then, with the notation
$l'=\sum_j n_jE_j^*\in \cS'$, set
\begin{equation}\label{zht}
Z(\bt):=\sum_{l'\in\cS'} \big(\prod_{j\in\cV}\chi_{\delta_j,n_j}
\big)\cdot \bt^{l'}.
\end{equation}
Since $Z=\Phi(\sum_{\bn\geq 0}(\prod_{j} \chi_{\delta_j,n_j}){\bf
x}^{{\bf n}})=\Phi(z)$, one gets:

\begin{proposition}\label{TH1}
The Taylor expansion at the origin (in variables $\bt^{E^*_i}$)
 of $\zeta$ is the series  $Z(\bt)$.
\end{proposition}

\subsection{} Next, we present a new (to the best of the author's
knowledge), and more subtle appearance of $\zeta(\bt)$. It uses
the euler-characteristic of cycles supported by $E$, hence (as we
will see) it creates the right bridge connecting topology with
sheaf-theoretical invariants. In order to state it, we need some
preparation.

As usual, let $K_{\widetilde{X}}\in L'$ be the {\em canonical
cycle}  associated with $\pi$. It is identified by the equations
$(K_{\widetilde{X}}+E_j,E_j)=-2$ for any $j$.  The (Riemann-Roch)
euler-characteristic $\chi(l)=-(l,l+K_{\widetilde{X}})/2$ can be
extended to any $l'\in L_\setQ$ by
$\chi(l')=-(l',l'+K_{\widetilde{X}})/2$. Clearly,
\begin{equation}\label{chiadd}
\chi(l'_1+l'_2)=\chi(l'_1)+\chi(l'_2)-(l'_1,l'_2).
\end{equation}

For any subset $I\subset\cV$ we write $E_I:=\sum_{j\in I}E_j$.

\begin{deflem}\label{DL}
\begin{enumerate}
\item\label{DL1}
For any fixed $l'\in L'$ and subset $I\subset \cV$, there is a
unique minimal subset $\bI\subset \cV$ which contains $I$, and
such that
\begin{equation}\label{star}
\mbox{there is no $j\in \cV\setminus \bI$ with
$(E_j,l'+E_\bI)>0$.}\end{equation}

\item\label{DL2} $\bI$ can
be found by the following algorithm: one constructs a sequence
$\{I_m\}_{m=0}^k$ of subsets of $\cV$, with $I_0=I$,
$I_{m+1}=I_m\cup \{j(m)\}$, where the index $j(m)$ is determined
as follows. Assume that $I_m$ is already constructed. If $I_m$
satisfies (\ref{star}) we stop and $m=k$. Otherwise, there exists
at least one $j$ with $(E_j,l'+E_{I_m})>0$. Take $j(m)$ one of
these $j$ and continue the algorithm  with $I_{m+1}$. Then
$I_k=\bI$.
\end{enumerate}
\end{deflem}
\begin{proof}
For (\ref{DL1}) notice that if $J_1$ and $J_2$ satisfies the
wished requirement (\ref{star}) of $\bI$ then $J_1\cap J_2$
satisfies too. The part (\ref{DL2}) is a version of the well-known
Laufer algorithm, compare  e.g. with (\ref{fact}).
\end{proof}

\begin{theorem}\label{PROPTOP}
\begin{equation}\label{EQ:100}
\sum_{l'\in\cS'}\sum_I
(-1)^{|I|+1}\chi(l'+E_\bI){\bt}^{l'}=Z(\bt),
\end{equation}
where the sum $\sum_I$ runs over all the subsets $I$ of \,$\cV$.
In other words, with the notation $l'=\sum_jn_j E^*_j$, one has
\begin{equation}\label{uj}
\sum_{\bn\geq 0}\sum_I(-1)^{|I|+1}\chi(l'+E_{J(l',I)})\bx^\bn=z(\bx)
\end{equation}
(where $\bx^\bn= x_1^{n_1}\cdots x_s^{n_s}$), or, for any  $l'=\sum_jn_j E^*_j$
$$\sum_I(-1)^{|I|+1}\chi(l'+E_\bI)=\prod_j\chi_{\delta_j,n_j}.$$
\end{theorem}
\begin{proof}  In the proof we use induction over $s=|\cV|$. If
$s=1$ then the identity is elementary, whose verification is left
to the reader. Hence, we assume $s>1$. Fix a vertex in
$\cV=\{1,\ldots,s\}$ corresponding to the index $s$ (after  a
possible  reordering) so that $\delta_s=1$. Let
$\Gamma_0:=\Gamma\setminus \{s\}$, and let $j_0$ be that vertex of
$\Gamma_0$ which is adjacent to $s$ in $\Gamma$.

Let $f_\Gamma(\bx)$ denote the left  hand side of (\ref{uj}). We
wish to show that $f_\Gamma(\bx)=z_\Gamma(\bx)$. Let
$\bx_0:=(x_1,\cdots,x_{s-1})$. The induction start with the
identity
\begin{equation}\label{z0}
z_{\Gamma}(\bx)=z_{\Gamma_0}(\bx_0)\cdot \frac{1-x_{j_0}}{1-x_s}.
\end{equation}
We will establish similar identity for $f_\Gamma$.  For this we
will use the notation $l'_0:=\sum_{j<s} n_jE^{*,\Gamma_0}_j$ for
any $l'=\sum_j n_j E^*_j$ (here $E^{*,\Gamma_0}_j$ is the dual of
$E_j$ in $\Gamma_0$). Notice that for any $Z$ supported in $\Gamma_0$
one has
\begin{equation}\label{3.uj}
(l',Z)=(l'_0,Z) \ \ \mbox{and} \ \ (-E^{*,\Gamma_o}_{j_0},Z)=(E_s,Z).
\end{equation}

First, we fix some $l'\in\cS'$ and a subset $I\subset \cV$ with
$s\not\in I$. If $s\in\bI$, we may assume (cf. the notations of
(\ref{DL})) that $I_{k-1}=\bI\setminus s$. Since $(l',E_s)\leq 0$,
$\delta_s=1$ and $(l'+E_{I_{k-1}},E_s)>0$, we get that, in fact,
$(l'+E_{I_{k-1}},E_s)=1$. Hence
$\chi(l'+E_{I_k})=\chi(l'+E_{I_{k-1}})$.  In other words, if
$s\not\in I$, then
\begin{equation}\label{egy}
\chi(l'+E_{\bI})=\chi(l'+E_{\bI\setminus s}).\end{equation} Using
$E_{J(l',I)\setminus s}=E^{\Gamma_0}_{J(l'_0,I)}$, (\ref{3.uj})  and
(\ref{chiadd}), (\ref{egy}) can be rewritten in the following
identity, where in the right hand side all invariants are
considered in $\Gamma_0$:
\begin{equation}\label{Gegy}
\chi(l'+E_{\bI}^\Gamma)-\chi(l')=\chi(l'_0+E_{\bI}^{\Gamma_0})-\chi(l'_0).
\end{equation}
Next, fix again $l'\in\cS'$ and take $I\subset \cV$ with  $s\in
I$. By (\ref{3.uj}) one has
$J(l',I)=J^{\Gamma_0}(l_o',I\setminus s)\cup s$. Using this, 
 (\ref{chiadd}) and (\ref{3.uj}) with a computation we get
\begin{equation}\label{Gketto}
\chi(l'+E_\bI^\Gamma)-\chi(l'+E_s)=\chi(l'_0-E^{*,\Gamma_0}_{j_0}+
E^{\Gamma_0}_{J(l'_0,I\setminus
s)})-\chi(l'_0-E^{*,\Gamma_0}_{j_0}).
\end{equation}
Since for any constant $c$, one has $\sum_{I:I\not\ni s}(-1)^{|I|+1}c
=\sum_{I:I\ni s}(-1)^{|I|+1}c=0$,
the identities (\ref{Gegy}) and (\ref{Gketto}) read as
\begin{equation*}
\sum_{\bn\geq 0}\sum_{I\not\ni
s}(-1)^{|I|+1}\chi(l'+E_{J(l',I)})\bx^\bn=f_{\Gamma_0}(\bx_0)\cdot
\sum_{n_s\geq 0} x_s^{n_s};
\end{equation*}
\begin{equation*}
\sum_{\bn\geq 0}\sum_{I\ni
s}(-1)^{|I|+1}\chi(l'+E_{J(l',I)})\bx^\bn=-f_{\Gamma_0}(\bx_0)
x_{j_0}\cdot \sum_{n_s\geq 0}x_s^{n_s}.
\end{equation*}
Hence $f_\Gamma(\bx)=f_{\Gamma_0}(1-x_{j_0})\sum_{n_s\geq
  0}x_s^{n_s}$. 
\end{proof}

\section{The analytic invariant}

\subsection{The setup}\label{setup} We start with a normal surface singularity
$(X,o)$, and we fix one of its good resolutions $\pi$.  In the
sequel $L$ and $L'$ will stay for the corresponding lattices
associated with $\pi$. Moreover, $L$ (respectively $L_\setQ$) will
also be identified with integral (rational) divisorial cycles
supported by
$E$. 

We denote by $c:(Y,o)\to (X,o)$ the universal abelian cover of
$(X,o)$. Let $\widetilde{Y}$ be the normalized pull-back of
$\widetilde{X}$ by $c$, $\pi_Y:\widetilde{Y}\to Y$ the pull-back
of $\pi$, and $\widetilde{c}: \widetilde{Y}\to \widetilde{X}$ the
induced finite map making the diagram commutative. We write
$\widetilde{c}^*$ for the pull-back of integral/rational cycles.

By \cite[(3.3)]{Line}, $\widetilde{c}^*(l')$ is an {\em integral}
cycle for any $l'\in L'$.

\begin{definition}\label{filtr}
We define a filtration on the local ring of holomorphic functions
$\cO_{Y,o}$:  for any $l'\in L'$, we set
\begin{equation}\label{eq:F1}
\cF(l'):=\{ f\in \cO_{Y,o}\ | \ div(f\circ \pi_Y)\geq
\widetilde{c}^*(l')\}.
\end{equation}
Notice that the natural action of $H$ on $Y$ induces an action on
$\cO_{Y,o}$ which keeps  $\cF(l')$ invariant. Therefore, $H$ acts
on $\cO_{Y,o}/\cF(l')$ too. For any $l'\in L'$, let
$\hh_{[l']}(l')$ be the dimension of the $\theta([l'])$-eigenspace
$(\cO_{Y,o}/\cF(l'))_{\theta([l'])}$. Then one defines the Hilbert
series \,$H(\bt)$ by
\begin{equation}\label{eq:33}
H(\bt):=\sum_{l'\in L'} \hh_{[l']}(l')\bt^{l'}.
\end{equation}
\end{definition}

\begin{remark}\label{P} Campillo, Delgado and Gusein-Zade in
\cite{CDG,CDGEq} use two other series as well. Here we present their
relationship  with the above Hilbert series. Set
\begin{equation}\label{eq:2}
L(\bt):=\sum_{l'\in L'}\dim_\setC
 (\cF(l')/\cF(l'+E))_{\theta([l'])}\cdot \bt^{l'}\in
\setZ[[\bt^{L'}]],
\end{equation}
where $E=\sum_jE_j$ as above. Since $\cF(l')=\cF(l'+E_j)$ if
$(l',E^*_j)<0$, it is easy to see that
$L(\bt)\prod_j(t_j-1)$ is an element of $\setZ[[\bt^{L'_e}]]$.
Hence the next infinite power series is well-defined:
\begin{equation}
P(\bt):= -\frac{L(\bt)\prod_j(t_j-1)}{1-\bt^E}=
-L(\bt)\prod_j(t_j-1)\cdot \sum_{k\geq 0} \bt^{kE}\in
\setZ[[\bt^{L'_e}]].
\end{equation}
Notice that $\hh_h(l')=0$ for $l'\leq 0$. This and the obvious
relation
$$H(\bt)(1-\bt^E)=\bt^E\cdot L(\bt)$$
show that
\begin{equation}\label{eq:4}
P(\bt)=-H(\bt)\cdot \prod_j(1-t_j^{-1}) \ \ \mbox{in
$\setZ[[\bt^{L'_e}]]$}.
\end{equation}
\end{remark}

Apparently,  taking  $P$ instead of $H$, one loses some analytic
information of $H$ (because $\prod_j(1-t_j^{-1})$ is a zero divisor).
Nevertheless, in \cite{CDG} there is an argument that, in fact, 
$P$  and $H$ contain the same amount of information (without providing
a formula for the inversion). 

In the next lemma, we show that this can be done very explicitly. In order
to state the expression, first we have to notice that, in fact, $P(\bt)$ is
supported on $\cS'$ (this is an automatically consequence of the 
discussion from (\ref{REF}), but can be verified directly as well). 

\begin{lemma}\label{lem:new}\cite{BN_splice1}
Write $P(\bt)=\sum_{l'\in   \cS'}\pp(l')\bt^{l'}$ for its coefficients. 
Then:
$$H(\bt)=P(\bt)\cdot \sum_{a\in L,\, a\not\geq 0}\bt^{-a}.$$
In other word, 
\begin{equation}\label{eq:sum}
\hh_{[l']}(l')=\sum_{a\in L,\, a\not\geq 0} \,\pp(l'+a).
\end{equation}
\end{lemma}
\noindent (The elements of $\cS'$ have the property that if one of the coefficients
tends to infinity then all of them do, hence the sum in (\ref{eq:sum}) is
finite.)  For the proof  see (\ref{rem.new}).

\subsection{Reformulation  of $H(\bt)$ and $P(\bt)$.}\label{REF}
The integers $\hh_h(l')$ can be reinterpreted in terms of
sheaf-cohomology. We start with the following result (for the
definition of $r_h$ see (\ref{cones})):

\begin{proposition}\label{fact}
\begin{enumerate}
\item\label{Fact1} \cite[(3.7)]{Line} \cite[(3.1)]{OkumaRat} For
any $h\in H$ there exists a divisor $D_h$ on $\widetilde{X}$
numerically equivalent to $r_h$ (i.e.
$c_1(\cO_{\widetilde{X}}(D_h))=r_h$) such that
\begin{equation*}
\widetilde{c}_*\cO_{\widetilde{Y}}=\bigoplus_{h\in H}\
\cO_{\widetilde{X}}(-D_h),
\end{equation*}
where the last sum is an $H$-eigenspace decomposition:
$\cO_{\widetilde{X}}(-D_h)$ is the $\theta(h)$-eigenspace of \,
$\widetilde{c}_*\cO_{\widetilde{Y}}$. Moreover, the Chern class
$c_1:Pic(\widetilde{X})\to L' $, $D\mapsto
c_1(\cO_{\widetilde{X}}(D))$, admits a group section $d_1:L'\to
Pic(\widetilde{X})$ given by
$d_1(l+r_h)=\cO_{\widetilde{X}}(l+D_h)$ (where $l\in L$ is
identified with an integral cycle).
\item\label{fact2} \cite[(4.2)]{Line} For any $l'\in L'$ there
exists a unique minimal $s(l')\in \cS'$ such that $l'\leq s(l')$ and
$[l']=[s(l')]$.

$s(l')$ can be computed by the following (Laufer) algorithm. One
constructs a `computation sequence' $x_0,\ldots, x_k\in L$ with
$x_0=0$ and $x_{m+1}=x_m+E_{j(m)}$, where the index $j(m)$ is
determined by the following principle. Assume that $x_m$ is
already constructed. If $l'+x_m\in\cS'$ then one stops and $k=m$.
Otherwise, there exists at least one $j$ with $(l'+x_m,E_j)>0$.
Take for $j(m)$ one of these $j$, and continue with $x_{m+1}$. The
algorithm stops after a finitely many
steps, and $l'+x_k=s(l')$. 
\end{enumerate}
\end{proposition}
Since for any $f\in \cO_{Y,o}$ that part of $div(f\circ \pi_Y)$
which is supported by the exceptional divisor of $\pi_Y$ is in the
Lipman cone of $\widetilde{Y}$, we get that $\cF(l')=\cF(s(l'))$.

On the other hand, for any $l'>0$ (in particular for any non-zero
$s(l')$), one has the exact sequence
\begin{equation}\label{exSeq}
0\to \cO_{\widetilde{Y}}(-\widetilde{c}^*(l'))\to
\cO_{\widetilde{Y}}\to \cO_{\widetilde{c}^*(l')}\to 0.
\end{equation}
If we consider (\ref{exSeq}) for some $l'=l+r_h>0$ ($l\in L$),
then by \cite[(3.2)]{OkumaRat} the corresponding
$\theta(h)$-eigenspaces constitute the exact sequence
\begin{equation}\label{exSeqphi}
0\to \cO_{\widetilde{X}}(-D_h-l)\to \cO_{\widetilde{X}}(-D_h)\to
\cO_{l}(-D_h)\to 0.
\end{equation}
In particular, we get
\begin{corollary}\label{4.cor.new} For any $l'=l+r_h>0$ one has:
$$\hh_h(l')=\dim\frac{H^0(\widetilde{X},\cO_{\widetilde{X}}(-D_h))}
{H^0(\widetilde{X},\cO_{\widetilde{X}}(-D_h-l))}.$$
\end{corollary}

\begin{example}\label{ex: h0}
Since $D_0=0$, for  $h=0$ we obtain: 
\begin{equation*}
H_0(\bt)=\sum_{l\in L}\dim \frac{\cO_{X,o}}{\{f\in \cO_{X,o}:
div(f\circ \pi) \geq l\}}\ \bt^l.
\end{equation*}
This is the Poincar\'e (Hilbert) series of $\cO_{X,o}$ associated
with the divisorial filtration of $\pi$. 
For some of its qualitative   properties
and interesting examples, see e.g. \cite{CHR}.
\end{example}

Therefore,  for any $l'=l+r_h>0$, from (\ref{exSeqphi}),  one has:
\begin{equation}\label{EXS0}
\hh_h(l')=\chi(\cO_l(-D_h))+ h^1(-D_h)-h^1(-D_h-l),
\end{equation}
where $h^1(D)$ denotes the dimension of
$H^1(\cO_{\widetilde{X}}(D))$. It is also convenient to write
$h^1(l')$ for $h^1(d_1(l'))$, cf. (\ref{fact})(\ref{Fact1}), e.g.
$h^1(-D_h-l)=h^1(-l')$.
Moreover,  $\chi(\cO_l(-D_h))=\chi(l')-\chi(r_h)$ for any
$l'=l+r_h$. Hence, (\ref{EXS0}) reads as 
\begin{equation}\label{EXS0b}
\hh_h(l')+h^1(-l')=\chi(l')-\chi(r_h)+ h^1(-r_h).
\end{equation}
Notice that the right hand side 
is a quadratic function in $l$.
Since  $\hh_h(l')=\hh_h(s(l'))$ we also get 
\begin{equation}\label{EXS}
\hh_h(l')=
\chi(s(l')) - \chi(r_h) -h^1(-s(l'))+ h^1(-r_h).
\end{equation}
Hence, by (\ref{eq:4}), 
\begin{equation}\label{pbar}
P(\bt)=\sum_{l'}\sum_I(-1)^{|I|+1}\, \Big(\,\chi(s(l'+E_I))
-h^1(-s(l'+E_I))\,\Big) \, \bt^{l'}.
\end{equation}

Assume that $l'\not\in\cS'$. Then there exists $j_0\in \cV$ with
$(l',E_{j_0})>0$. Therefore, for any subset $I\subset \cV\setminus
\{j_0\}$ the Laufer algorithm applied for $l'+E_I$ may start
adding $E_{j_0}$, hence $s(l'+E_I)=s(l'+E_{I\cup j_0})$. In
particular, these two terms cancel each other in the sum of the
above expression of $P$. In particular, the first sum in
(\ref{pbar}), in fact, runs over $l'\in \cS'$.

Moreover, analysing the computation sequence $\{x_i\}_i$ connecting
$l'+E_I$ with $s(l'+E_I)$, at each step one has
\begin{equation*}
\chi(x_m)-h^1(-x_m)=\chi(x_{m+1})-h^1(-x_{m+1}).
\end{equation*}
Since $l'+E_{J(l',I)}$ is one of the elements of this sequence, we get:
\begin{corollary}\label{pbarcor}
\begin{align*}
P(\bt)&=\sum_{l'\in \cS'}\sum_I(-1)^{|I|+1}\, \Big(\,\chi(s(l'+E_I))
-h^1(-s(l'+E_I))\,\Big) \, \bt^{l'}\\
\ &=\sum_{l'\in \cS'}\sum_I(-1)^{|I|+1}\, \Big(\,\chi(l'+E_{J(l',I)})
-h^1(-l'+E_{J(l',I)})\,\Big) \, \bt^{l'}\\
\ &=\sum_{l'\in \cS'}\sum_I(-1)^{|I|+1}\, \Big(\,\chi(l'+E_I)
-h^1(-l'+E_I)\,\Big) \, \bt^{l'}.
\end{align*}
\end{corollary}
\noindent Via (\ref{EQ:100}), the second identity reads as
\begin{corollary}
\begin{equation*}
P(\bt)=Z(\bt)+
\sum_{l'\in \cS'}\sum_I(-1)^{|I|}\, 
h^1(-l'+E_{J(l',I)})\, \bt^{l'}.
\end{equation*}
I.e., $P(\bt)=Z(\bt)$ if and only if 
$\sum_I(-1)^{|I|} h^1(-l'+E_{J(l',I)})=0$ for all $l'\in \cS'$.
\end{corollary}


\begin{remark}\label{rem.new}
By (\ref{eq:4})
$$\pp(l')=\sum_I(-1)^{|I|+1}\dim\frac{H^0(\cO(-r_h))}
{H^0(\cO(-l'-E_I))}=
\sum_I(-1)^{|I|+1}\dim\frac{H^0(\cO(-l'))}
{H^0(\cO(-l'-E_I))}.$$
Using the very last expression for $\pp(l')$, one can show 
that for any $l\in L_{\geq 0}$, one has (cf. \cite{BN_splice1}):
$$\dim\frac{H^0(\widetilde{X},\cO(-l'))}
{H^0(\widetilde{X},\cO(-l'-l))}=\sum_{a\in L_{\geq 0},\, a\not\geq l}
\pp(l'+a).$$
The proof uses induction over $l$, one compares the above expressions
for $l$ and $l+E_i$ (for some $i$). Then, taking $l'=r_h$, one gets
the proof of (\ref{lem:new}).
\end{remark}

\section{The Main Identity.}\label{mi}

\subsection{} In the next paragraph, we formulate an identity
which connects the topological invariant $Z$ with the analytic
invariant $P$, which is conjecturally valid for `some'   normal
surface singularities.
\begin{definition}\label{MI} Consider a normal surface singularity $(X,o)$
with rational homology sphere link. We say that $(X,o)$ and its
resolution $\pi$ satisfy
 the {\bf Main Identity} if
\begin{equation}\label{MIF}
Z(\bt)= P(\bt).
\end{equation}
\end{definition}

\begin{example}\label{segy}
 Consider the cyclic quotient singularity whose minimal
resolution $\pi$ has only one irreducible component $E$ with
self-intersection $-p$ ($p\geq 2$). Then $H=\setZ_p$, $Y=\setC^2$,
$\cO_{Y,o}=\setC\{z_1,z_2\}$. The action of $H$ on $\setC\{z\}$ is
given by $h*z_i=\theta(h)(E^*)z_i$. Fix the generator $g:=[-E^*]$
of $H$ and set $\xi:=e^{2\pi i/p}$. Then the action of $g$ is
given by $g*z_i=\xi z_i$. Fix $\phi\in \widehat{H}$ with
$\phi(g)=\xi^q$ for some $0\leq q<p$. Then $\phi=\theta(h)$ with
$h:=[qE^*]$. Moreover $r_h=qE^*=(q/p)E$.

$\widetilde{Y}$ is just the blow up of $Y$ in one point with
exceptional divisor $\widetilde{E}$ a $(-1)$-curve. Since
$\widetilde{c}^*(E)=p\widetilde{E}$, we get that $\cF(l')$
consists of all the monomials with degree $\geq pl'$. Therefore,
for $l'=l+q/p$ ($l\in L$), $(\cF(l')/\cF(l'+E))_\phi$ can be
identified with the vector space of monomials of degree $pl+q$.
Hence
$$P_h(t)=\sum_{l\geq 0}(1+q+pl)t^{l+\frac{q}{p}}, \ \mbox{and} \
P(t)=\sum_{k\geq 0}(1+k)t^{k/p}.$$ On the other hand, $Z(t)$
equals with the same sum by its very definition.
\end{example}

\begin{remark}
Notice that if for some analytic structure and resolution 
$Z=P$ holds, then, 
by (\ref{lem:new}), the Hilbert function $H$ can be recovered from 
the resolution graph of $\pi$.
\end{remark}

The Main Identity for the component $h=0$ (cf. \ref{hcompa}) was
proved in \cite{CDG} for any rational singularity and any
resolution, see also \cite{CDGEq} for an equivariant version,
valid for rational singularities, formulated in a different way. 
Here we prove it in the following situations:

\begin{theorem}\label{MTH} The Main Identity is true 
in the following cases:
\begin{enumerate}
\item\label{rat}
$(X,o)$ is rational, and  $\pi$ is arbitrary resolution,
\item\label{me}
$(X,o)$ is minimally elliptic singularity whose minimal resolution
is good, and $\pi$ is this minimal resolution,
\end{enumerate}
\end{theorem}

The present proof emphasizes the importance of some vanishing
theorems valid for these singularities, and also the numerical
behavior of the `computation sequences' associated with the Artin
cycles.

\vspace{2mm}

\noindent \emph{Proof of (\ref{MTH}).} Let $p_g$ be the
geometric genus of $(X,o)$. Then $(X,o)$ is rational iff $p_g=0$,
and $(X,o)$ is minimally elliptic iff $p_g=1$ and $(X,o)$ is
Gorenstein \cite{Laufer77}. We write $Z_{min}\in L$ for the Artin
cycle,
 the unique minimal element of $\cS\setminus {0}$.

By \cite[1.7]{rohr}, one gets $h^1(-s(l'))=0$ except when $(X,o)$
is minimally elliptic and  $s(l')=0$ (in the rational case this
vanishing was proved by several authors, see e.g.,
\cite[12.1]{li}, \cite[4.1.1]{Line}). On the other hand, $s(l')=0$
if and only if $l'\in -L_e$.  Therefore, the second term in
(\ref{pbar}) is
\begin{equation*}
\Big(\, \sum_{l'}h^1(-s(l'))\bt^{l'}\, \Big) \cdot \prod_j
(1-t_j^{-1})= \Big(\, \sum_{l\in -L_e} p_g\bt^l \, \Big) \cdot
\prod_j (1-t_j^{-1})=p_g\bt^0.
\end{equation*}
Hence, (\ref{pbar}) (or (\ref{pbarcor})) reads as
\begin{equation}\label{pbar2}
P(\bt)=p_g\bt^0+\sum_{l'\in\cS'}
\sum_I(-1)^{|I|+1}\chi(s(l'+E_I))\cdot \bt^{l'}.\end{equation}

\begin{remark} Notice that (\ref{pbar2}) (as a consequence of
the above vanishing result) already shows that $P(\bt)$ is topological
(i.e. it can be determined from $\Gamma$).

In general, the above vanishing is not valid (even for `very
large' $l$). E.g., for Gorenstein elliptic singularities with
$p_g>1$, $h^1(\cO_{\widetilde{X}}(-kZ_{min}))=p_g-1>0$ for any
$k>0$ \cite[2.21]{Ninv}.

Moreover, in general, $P$ is not topological, it might depend on
the choice of the analytic structure (supported by a fixed
topological type), see e.g. (\ref{me2}).
\end{remark}
In the sum (\ref{pbar2}), let us separate $l'=0$. In this case,
since $s(E_I)=0$ for $I=\emptyset$, and $=Z_{min}$ otherwise, we
get  $\sum_I(-1)^{|I|+1}\chi(s(E_I))=\chi(Z_{min})=1-p_g$. In
particular
\begin{equation}\label{pbar3}
P(\bt)=\bt^0+\sum_{l'\in\cS'\setminus 0}
\sum_I(-1)^{|I|+1}\chi(s(l'+E_I))\cdot \bt^{l'}.\end{equation}

Now we focus on the sum in the right-hand side of this identity.
If we run the Laufer algorithm (\ref{fact})(\ref{fact2}) for
$l'+E_I$, then for any $m$ we have
\begin{equation}\label{alg}
\chi(l'+x_{m+1})=\chi(l'+x_{m})+1-(l'+x_m,E_{j(m)}).
\end{equation}

\begin{remark}\label{egyketto}
Take $h=0$. If in the Laufer  algorithm (\ref{fact})(\ref{fact2})
one starts with $l=E$, one gets for $s(l)$ the Artin cycle
$Z_{min}$ \cite{Laufer72}. During these steps, in the rational
case one has all the time $(E+x_m,E_{j(m)})=1$ \cite{Laufer72}. On
the other hand, in the minimally elliptic case one has all the
time $(E+x_m,E_{j(m)})=1$ excepting in the very last step, when
$(E+x_m,E_{j(m)})=2$; moreover, for any prescribed $E_{i^*}$ one
may arrange the computation sequence in such a way, that in this
very last step $j(k-1)=i^*$, cf.  \cite{Laufer77}.
\end{remark}

\begin{lemma}\label{L:1} Assume that $(X,o)$ and $\pi$ is as in
(\ref{MTH}). The for any $l'\in\cS'\setminus 0$ and $I\subset \cV$
one has
$$\chi(s(l'+E_I))=\chi(l'+E_{\bI}).$$
\end{lemma}
\begin{proof} Write $l'=l+r_h$.
In the Laufer algorithm of $l'+E_I$ we may start the sequence with
cycles of type $l'+E_{I_m}$, where $I_{m+1}=I_m\cup \{i_m\}$ and
$i_m\not\in I_m $ (i.e. with the computation sequence of
$J(l',I)$, cf. (\ref{DL})(\ref{DL2})), till we hit $l'+E_{\bI}$.
Then, as a second turn, we continue to add cycles $E_j$ with $j\in
\bI$ till we get the smallest cycle $z\in L'$ with $z\geq
l'+E_{\bI}$ and $(z,E_j)\leq 0$ for all $j\in \bI$. Let $Z^*$ be
the (sum of the) Artin cycles of (the connected components of)
$\bI$. Clearly $z\leq l'+Z^*$. Moreover, the steps in this second
turn may be considered as a part of the computation of the Artin
cycle $Z^*$.

Assume $\bI=\cV$. Then $z=s(l'+E_I)$. Moreover
 $\chi(z)=\chi(l'+E_{\bI})$.
 This follows from  \eqref{alg} and \eqref{egyketto}:
in the second turn we have all the time $(l'+x_m,E_{j(m)})=1$
except if the graph is elliptic, $z=l'+Z_{min}$, and we are in the
very last step. In this case,  we may assume that the last added
cycle $E_j$ is one prescribed with $(l',E_j)<0$, hence
$(z-E_j,E_j)=1$ too.

Assume that $\bI\not=\cV$. Then, by \cite{Laufer77}, $\bI$
supports a rational singularity, hence by \eqref{alg} and
\eqref{egyketto},   $\chi(z)=\chi(l'+E_{\bI})$. If $z=s(l'+E_I)$
then we are done. Otherwise, there exists an index $j\in
\cV\setminus \bI$ with $(z,E_j)>0$ ($*$). Notice that
automatically $(E_{\bI},E_j)>0$ (otherwise we would have
$(z,E_j)=(l',E_j)\leq 0$, a contradiction). By the definition of
$\bI$, we have $(l'+E_{\bI},E_j)\leq 0$, hence $(l',E_j)\leq -
(E_{\bI},E_j)<0$. This together with ($*$) implies $(z-l',E_j)\geq
2$. Since the computation sequence $\{x_m\}_m$ can be considered
as the beginning of the computation sequence of $Z_{min}$, this
can happen only if the graph is elliptic, $(z-l',E_j)=2$,
$z+E_j=Z_{min}+r_h=s(l'+E_I)$ and $z=l'+Z^*$. But in this case one
also has $0<(z,E_j)=2+(l',E_j)<2$, hence $(z,E_j)=1$ which imply
$\chi(z)=\chi(z+E_j)$.
\end{proof}
Finally, (\ref{MTH}) follows from (\ref{PROPTOP}), (\ref{pbar3})
and  (\ref{L:1}) (and the fact that the coefficient of $\bt^0$ in
$Z(\bt)$ is 1). \hfill$\square$

\section{The Reduced Identities}\label{sec:RI}

\subsection{} For any fixed resolution $\pi$,
in the main identities (\ref{MI}) one takes a variable $t_j$ for
each exceptional divisor $E_j$ of $\pi$. Since, in general, $\pi $
is not minimal, (and the set of all exceptional divisors of an
arbitrary resolution supports no intrinsic geometric meaning), it
is natural to create the possibility to reduce the number of
variables, or replace $L$ and $L'$ by some more intrinsic
(geometrically defined)  lattices. E.g., one might take only the
vertices corresponding to the nodes of the graph, or corresponding
to the essential exceptional divisors only, or pull-backs/images of lattices
provided by some geometric/universal constructions.

Here, we make explicit that  situation when one 
reduces the number of variables. To start with, we define the
corresponding topological and analytical series with reduces
number of variables, then we formulate the corresponding `Reduced
Identities', and we prove that the `Main Identity' implies all the
reduced ones. Finally, in the next section, we interpret some
classical and recent relations of the singularity theory as
reduced identities.

We fix $(X,o)$ and the resolution  $\pi$ as in (\ref{setup}). Let
$\cU$ be a non-empty subset of 
$\cV$. Our plan is to define formal series in variables
$\{t_j\}_{j\in \cU}$, still denoted by $\bt$. It is convenient to
define the projection $pr_\cU:L'\to L'$ by $pr_\cU(\sum_jl'_jE_j)=
\sum_{j\in \cU}l'_jE_j$. Correspondingly, we write
\begin{equation*}
\bt^{pr_\cU(l')}=\prod_{j\in \cU}t_j^{l'_j}=\bt^{l'}|_{t_j=1\,
\mbox{\tiny{for all}} \,j\not\in \cU}.\end{equation*}

Although, in the non-reduced case, $\bt^{l'}$ codifies completely
the $\theta(h)$-eigenspace which provides the coefficient of
$\bt^{l'}$ (via $h=[l']$), this is not true anymore for
$\bt^{pr_\cU(l')}$. Hence, in the reduced context we lose this
correspondence, and we do not consider $h$-components in the sense
of (\ref{hcompa}).
\begin{definition}
For any $h\in H$ set
\begin{equation}\label{zeta2r}
\zeta_h^\cU(\bt) \coloneqq \frac{1}{\lvert H \rvert} \cdot
\sum_{\rho \in \widehat{H}}\, \rho(h)^{-1}\cdot \prod_{j\in\cV}
{(1-\rho([E^*_j]) \bt^{pr_\cU(E^*_j)})}^{\delta_j-2},
\end{equation}
or, equivalently,
\begin{equation}\label{zeta2r2}
\zeta_h^\cU(\bt)\coloneqq  \zeta_h(\bt)|_{t_j=1\, \emph{\tiny{for
all}} \,j\not\in \cU}.
\end{equation}
\end{definition}
Having in mind (\ref{TH1}), we will identify 
$\zeta^\cU_h$ and $Z^\cU_h$ (where, one might define 
$Z^\cU_h(\bt):=Z_h(\bt)|_{t_j=1\, \mbox{\tiny{for all}} \,j\not\in
\cU}$).

Next, we define the  analytic invariants. For any $l'_\cU\in
pr_\cU(L')\subset L'$ set
\begin{equation*}
\cF^\cU(l'_\cU):=\{ f\in \cO_{Y,o}\ | \ div(f\circ \pi_Y)\geq
\widetilde{c}^*(l'_\cU)\}.
\end{equation*}
Moreover, for any $l'_\cU\in pr_\cU(L')$ and $h\in H$, let
$\hh^\cU_h(l'_\cU)$ be the dimension of the $\theta(h)$-eigenspace
$(\cO_{Y,o}/\cF^\cU(l'_\cU))_{\theta(h)}$. Then set
\begin{align*}
H^\cU_h(\bt)&:=\sum_{l'_\cU\in pr_\cU(L')}\hh^\cU_h(l'_\cU)\cdot\bt^{l'_\cU},\\
 L^\cU_h(\bt)&:=\sum_{l'_\cU\in pr_\cU(L')}\dim
\Big(
\frac{\cF^\cU(l'_\cU)}{\cF^\cU(l'_\cU+E_\cU)}\Big)_{\theta(h)}\cdot
\bt^{l'_\cU},\\
P^\cU_h(\bt)&:= -\frac{L^\cU_h(\bt)\prod_{j\in
\cU}(t_j-1)}{1-\bt^{E_\cU}}.
\end{align*}
Then, again,
\begin{equation}\label{eq:4red}
P^\cU_h(\bt)=-H^\cU_h(\bt)\cdot \prod_{j\in \cU}(1-t_j^{-1}).
\end{equation}

\begin{definition}\label{MIRed}
 We say that $(X,o)$ and the resolution $\pi$ satisfy
 the {\bf Reduced Identity} for $h\in H$, corresponding to the subset $\cU\subset \cV$, if
\begin{equation}\label{MIFRed}
Z^\cU_h(\bt)= P^\cU_{h}(\bt)  \ \ \ \mbox{(\,in\
$\setZ[[t_j^{1/d},j\in\cU]]$, \, or in
$\setZ[[\bt^{pr_\cU(L')}]]$)}.
\end{equation}
We say that $(X,o)$ and $\pi$ satisfy the {\bf Reduced Identity} (for
$\cU$) if (\ref{MIFRed}) is true for all $h\in H$.
\end{definition}
Although, by (\ref{zeta2r2}), $Z^\cU$ can be
obtained by a simple substitution from $Z$, the same statement is
not immediate for $P^\cU$. Nevertheless, we have the following:

\begin{theorem}\label{MR}
\begin{equation}\label{reduc}
P_h^\cU(\bt)=P_h(\bt)|_{t_j=1\ \mbox{\tiny{for all}} \ j\not\in
\cU}.\end{equation} In particular, if for some singularity
$(X,o)$, resolution $\pi$ and element $h\in H$ the Main Identity
(\ref{MI}) is true, then for any non-empty $\cU\subset\cV$ the
Reduced Identity (for $(X,o)$, $\pi$ and $h$) is also true.
\end{theorem}
\begin{proof}
We prove (\ref{reduc}) by descending induction over the
cardinality of $\cU$. Set $\bar{\cU}\coloneqq \cU\setminus i_0$
such that $\bar{\cU}\not=\emptyset$. Then, using the above
identities, it is enough to verify
\begin{equation}\label{lbar}
L^\cU_h(\bt)(t_{i_0}-1)|_{t_{i_0}=1}=L^{\bar{\cU}}_h(\bt).\end{equation}
Let $L_\cU$ be the lattice generated by $\{E_j\}_{j\in \cU}$. Then
$L_h^\cU(\bt)\cdot \bt^{-pr_\cU(r_h)}$ has the form $\sum_{l\in
L_\cU}\ a^\cU_h(l)\ \bt^{l}$. This can be rewritten as
$$\sum_{\bar{l}\in L_{\bar{\cU}}} \ \bt^{\bar{l}}\
\sum_{k\in\setZ} \, a^\cU_h(\bar{l}+kE_{i_0})\, t_{i_0}^k.$$ By
definition,
\begin{equation*}
a^\cU_h(\bar{l}+kE_{i_0})=a^{\bar{\cU}}_h(\bar{l})\ \ \ \mbox{for
$k\ll 0$}.
\end{equation*}
Also, for any $l\in L$ there exists an integer $n$ such that if
for some $s\in \cS$ one has $s\geq l$ and $s_{i_0}\geq n$, then
automatically $s\geq l+E$ too. Therefore
\begin{equation*}
a^\cU_h(\bar{l}+kE_{i_0})=0\ \ \ \mbox{for $k\gg 0$}.
\end{equation*}
Hence
\begin{equation*}
\sum_{k\in\setZ} \, a^\cU_h(\bar{l}+kE_{i_0})\, t_{i_0}^k
(t_{i_0}-1)|_{t_{i_0}=1}= a^{\bar{\cU}}_h(\bar{l}).
\end{equation*}
\end{proof}

\section{Examples, counterexamples}

\subsection{} Theorem (\ref{MR}) connects the Main Identity
with some non-trivial results, already present in the literature.
In the beginning of this section we list some of them. Then, we
list some examples when $Z\not=P$, trying to catch the limits of
the Main Identity.

 In this section, we focus on the situation
when {\em $\cU$ contains only one element}, say $j_0$. Notice that
even in this case, the verification (or disprove) of the identity
(\ref{MIFRed}) for some specific analytic structures can be hardly
non-trivial.

The unique variable will be denoted by $t$; notice that
$P^\cU_h(t)=L^\cU_h(t)$. For the trivial element $0\in H$, the
filtration is the divisorial (valuation) filtration associated
with $E_{j_0}$ on $\cO_{X,o}$, and $L^\cU_0(t)$ is the Poincar\'e
series of the associated graded algebra.

\subsection{Example (weighted homogeneous singularities)}\label{qh} Assume
that $(X,o)$ is the germ at the
unique singular point of an affine space $X$ which admits a good
$\setC^*$ action (and  we also assume that it is not a cyclic
quotient singularity), and let $\pi$ be its minimal good
resolution. Then $\cO_{X,o}$ admits a grading $\oplus_{k\geq
0}A_k$ with Poincar\'e series $p_{X,o}(t)=\sum_{k}\dim A_kt^k$.
The link of such a singularity is a Seifert 3-manifold. In fact,
the dual graph of $\pi$ is star-shaped, which can be identified by
the Seifert invariants of the link as well. Assume that $\cU$
contains only one element, namely the central exceptional divisor
$E_{cen}$ of $\pi$.

For such a singularity, Pinkham in \cite{Pi1} proved that
$L_0^\cU(t)=p_{X,o}(t)$ and computed this Poincar\'e series  in
terms of the Seifert invariants of the link. On the other hand,
\cite[(5.2)]{nemethi04:_seiber_witten} identifies this expression
with $\zeta_0(t)$. In particular, for such a singularity, the
identity $P^\cU_0=Z^\cU_0$ follows.

The equivariant version $P^\cU_h=Z^\cU_h$ is proved by Neumann in 
\cite[\S4]{Neu} (for a more general argument, see
(\ref{sq})). From \cite{Neu}  one also obtains that
$\sum_hZ_h^\cU(t^c)$ is the Poincar\'e series $p_{Y,o}(t)$ of the
universal abelian cover $(Y,o)$ (where $c$ is the order of
$[E^*_{cen}]$ in $H$).

\subsection{Example (splice-quotient singularities)}\label{sq} Assume that $(X,o)$ is a
splice-quotient singularity (introduced and intensively studied by
Neumann and Wahl, see e.g. \cite{NWuj2}). We point here in short
what we mean by this.

We fix a possible resolution graph $\Gamma$. (It should satisfy
some combinatorial properties which are not essential in the
present discussion.) Then, using $\Gamma$, one can write equations
(up to an equisingular deformation) 
for the universal abelian cover $(Y,o)$ of $(X,o)$ and for the
action of $H$ on $(Y,o)$. These equations determine the analytic
structure of $(X,o)$. Then, the resolution used in all our discussions
is exactly that resolution $\pi_\Gamma$ of $(X,o)$ which has 
dual graph $\Gamma$ (it is  uniquely determined by $\Gamma$).
Notice that there is always a linkage between the analytic structures
constructed by this procedure and the graph used in the construction.  
The construction gives hope that the analytic invariants of the
singularity might be determined combinatorially from the dual
resolution graph, if one writes them (exactly) from $\pi_\Gamma$. 
Otherwise, this might not be the case, and the topology-analytic 
linkage might be lost, see also  (\ref{ex.7.new}).

We recall that splice-quotient singularities  include all the rational 
singularities (where $\Gamma$ can be arbitrary), all the 
minimally elliptic singularities (where $\Gamma $ has the property that
the support of the minimal elliptic cycle is not proper
--- e.g., for all minimally good resolutions)  \cite{Ouac-c}, 
and also those singularities which admit good $\setC^*$ action (and
$\Gamma$ is the minimal good resolution) \cite{Neu}.

Let us fix a graph $\Gamma$ and consider  the associated
splice-quotient singularity  and the resolution $\pi_\Gamma$. 
Set $\cU=\{j_0\}$, where $j_0$ is one of the nodes of $\Gamma$
(i.e. $\delta_{j_0}\geq 3$). Then the identity $Z_h^\cU=P_h^\cU$
was proved by Okuma in \cite[\S 3]{Opg}. (The same is true if $j_0$
is not a node, but is not an end-vertex either, and $H=1$, see
\cite{NO1}.)

In other words, by Okuma's result, the Main Identity reduced to
any individual node is true (in $\pi_\Gamma$).  
This is improved in \cite{BN_splice1},
where the Main Identity is proved in its whole generality.
(Notice that this very general result includes theorem (\ref{MTH})
too, nevertheless the proofs are very different, and the proof of 
(\ref{MTH}) might illuminate some aspects --- e.g. the relevance with vanishing
theorems ---, which are less transparent in \cite{BN_splice1}.)

\subsection{Example (hypersurface singularities)}\label{wh1}
This is a small digression about Poincar\'e polynomial
computations of (not necessarily weighted homogeneous)
hypersurface singularities when the filtration can be connected
with some weights.

Fix three positive integers ${\bf w}=(w_1,w_2,w_3)$, and consider
on $\cO:=\cO_{\setC^3,o}$ the graded ring structure $\oplus_{e\geq
0}\cO(e)$, where the three coordinates have degrees
$\deg(z_i)=w_i$. Here $\cO(e)$ is the vector space of monomials of
${\bf w}$-degree $e$.  Fix $f\in \cO$ and define
$(X,o):=\{f=0\}\subset (\setC^3,o)$. Write $f=f_d+f_{d+1}+\cdots$,
where $f_k$ is homogeneous of ${\bf w}$-degree $k$, and
$f_d\not=0$. Let $q:\cO\to\cO_{X,o}$ be the natural projection. 

We will analyze different  filtrations $\{\cF(l)\}_{l\geq 0}$ of the
$\setC$-algebra $\cO_{X,o}$. As usual, for any filtration $\cF$, we write
$Gr_l^\cF:=\cF(l)/\cF(l+1)$, and $P^\cF(t):= \sum_{l\geq
0}\dim(Gr_l^\cF)t^l$ for its Poincar\'e series.

First, consider the filtration $\cG$ defined by
$\cG(l):=q(\oplus_{e\geq l}\cO(e))$. The next isomorphism can be
easily verified and is left to the reader:

\begin{lemma}\label{wh2}
$$Gr^\cG_l \simeq \frac{\cO(l)}{f_d\cdot \cO(l-d)}.$$
In particular, $P^\cG(t)$ agrees with the Poincar\'e series of
$\cO/(f_d)$ with weights ${\bf w}$, namely
$P^\cG(t)=(1-t^d)/\prod_i (1-t^{w_i})$.
\end{lemma}

Next, we assume that $f$ is irreducible (this is automatically
satisfied if $(X,o)$ is normal). Let $K$ be the field of fraction
of $\cO_{X,o}$.  We are interested in filtrations determined by a
valuation $\nu:K^*\to \setZ$ with the property $\nu(z_i)=w_i$: we
set $\cF(l):=\{g\in \cO_{X,o}\,:\ g\not=0, \ \nu(g)\geq
l\}\cup\{0\}$. Clearly, from the definition of the valuations, we
have $\cG(l)\subset \cF(l)$ for all $l$. But the equality, in
general, fails.

\begin{remark}\label{wh3}
$\nu$ may appear as follows. Consider $\phi$, the ${\bf
w}$-weighted blow up of $(X,o)$,  and let $E$ be its strict
transform. (One might even take any resolution or modification
$\pi$ which dominates $\phi$, and identify $E$ with its
$\pi$-strict transform.) Assume that $E$ is irreducible. Then the
divisorial valuation $\nu$ of $E$ provides such a filtration
$\cF$.
\end{remark}

In the next lemma we restrict ourself to the case $w_3=1$
--- this is enough in the main application in (\ref{si});
we invite the reader to formulate the corresponding statement for
more general weights (when one must handle an additional cyclic
group action).

\begin{lemma}\label{wh4} Assume $w_3=1$ and $f_d$ irreducible
(but $f_d\not=z_3$). 
Consider the filtration $\cF$ as in (\ref{wh3}). Then $\cF=\cG$.
In particular, $P^\cF(t)=(1-t^d)/\prod_i(1-t^{w_i})$.

\end{lemma}

\begin{proof} We may do the computations in a chart of the weighted blow up
$\phi:Z\to \setC^3$. We fix the chart $(z_1,z_2,z_3)\mapsto
(\alpha^{w_1}\beta,\alpha^{w_2}\gamma,\alpha)$ of $\phi$. Then the
strict transform $\widetilde{X}$ of $X$ has local equations
$\{f_d(p)+\alpha f_{d+1}(p)+\cdots=0\}$, while the exceptional
curve $E\subset \widetilde{X}$ is $\{f_d(p)=\alpha=0\}$, where
$p=(\beta,\gamma,1)$.

We wish to show that $\cF(l)\subset \cG(l)$. For this, take $g$ with
$q(g)\in\cF(l)$, and write $g$ as a sum of ${\bf w}$-homogeneous elements
$g_k+g_{k+1}+\cdots $. If $k\geq l$ then $q(g)\in \cG(l)$, hence
assume that $k<l$. By substitution, $\phi^*(g)=\alpha^k h$, where
$h:=g_k(p)+\alpha g_{k+1}(p)+\cdots$. Since the vanishing order of
$\phi^*(g)$ is $\geq l$, $h$ must vanishes along $E$, hence there
exists local functions $u$ and $v$ in $(\alpha,\beta,\gamma)$ such
that $h=uf_d(p)+v\alpha$. Taking $\alpha=0$, we get a function
$\omega(\beta,\gamma)$
such that $g_k(\beta,\gamma,1)=\omega(\beta,\gamma)\cdot
f_d(\beta,\gamma,1)$. We interpret this in coordinates $z$. The 
homogenized identity has the form $g_k(z)=\tilde{\omega}(z)\cdot
f_d(z)\cdot z_3^a$ for some $a\in\setZ$. 
Since $f_d$ is irreducible, we obtain  $f_d|g_k$. Then
$g':=g-g_k+(g_k/f_d)(-f+f_d)$ has the property that $q(g')=q(g)$
but its $k$-homogeneous term vanishes. By induction, $g$ can be
replaced by $g''$ with $q(g'')=q(g)$ and $g''\in \oplus _{e\geq
l}\cO(e)$, hence $q(g)\in \cG(l)$.
\end{proof}

\begin{example}\label{237} Assume that $(X,o)$ is the zero set of
the unimodal exceptional hypersurface singularity of type
$E_{12}$: $f_a=z_1^2+z_2^3+z_3^7+ayz^5$. One may verify that
$(X,o)$ is minimally elliptic.  Let $\pi$ be its  minimal good
resolution, whose dual graphs is:

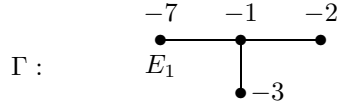
\begin{picture}(150,45)(-100,10)
\put(40,40){\circle*{4}} \put(70,40){\circle*{4}}
\put(100,40){\circle*{4}} \put(70,20){\circle*{4}}
\put(40,40){\line(1,0){60}} \put(70,40){\line(0,-1){20}}
\put(40,50){\makebox(0,0){$-7$}} \put(70,50){\makebox(0,0){$-1$}}
\put(100,50){\makebox(0,0){$-2$}} \put(80,20){\makebox(0,0){$-3$}}
\put(40,30){\makebox(0,0){$E_1$}}
\put(-10,30){\makebox(0,0){$\Gamma:$}}
\end{picture}

Let $E_{cen}$ be the node, and $\cU=\{E_{cen}\}$. Then
$Z^\cU=P^\cU$  by (\ref{qh}) if $a=0$, or by (\ref{sq}), in general.
Here we would like to show the same identity for $\cU=\{E_1\}$
(see the picture). Notice that $H$ is trivial. By a computation:
\begin{equation}\label{ZU}
Z^\cU(t)=\frac{t^6-1}{(t^3-1)(t^2-1)(t-1)}. \end{equation} On the
other hand, $P^\cU$ can be computed by (\ref{wh4}): the weights
are ${\bf w}=(3,2,1)$, hence  $P^\cU=Z^\cU$ follows.

Notice that $\pi$ is not minimal, hence (\ref{MTH}) cannot be
applied. The graph of all minimally elliptic singularities (graphs)
whose minimal resolution is not good are listed in \cite{Laufer77}. 
Those which provide rational homology sphere links are similar with
the above example, and for all of them by similar method on can show
that $Z=P$. (This fact follows also from \cite{BN_splice1},  where
$Z=P$ is proved for splice-quotients, since this applies for minimally
elliptic singularities and their minimal good resolutions, cf. (\ref{sq}).)
\end{example}

\subsection{Example (superisolated singularities)}\label{si} Let $(X,o)$ be a hypersurface
super-isolated singularity (a family introduced by I. Luengo in
\cite{Luengo}). More precisely, $(X,o)=(\{f=0\},o)\subset
(\setC^3,o)$, where $f=f_d+f_{d+1}+\cdots$, $f_k$ is homogeneous
of degree $k$, and the projective curve $\{f_{d+1}=0\}\subset
\setC\setP^2$ does not meet the singular points of $C\coloneqq
\{f_d=0\}\subset \setC\setP^2$. We assume $d\geq 2$ and  that {\em
$C$ is irreducible}. The link of $(X,o)$ is a rational homology
sphere if and only if $C$ is rational and cuspidal (i.e. all the
singular points $(C,p_i)_{i=1}^\nu$ of $C$ are locally
irreducible). In the sequel we assume these facts too.

The minimal resolution of $(X,o)$ can be obtained by a single blow
of the maximal ideal. The exceptional divisor (corresponding to
the tangent cone) is isomorphic to $C$. In general, this is not a
good resolution, one has to resolve the singularities of $C$ as
well. Hence, the minimal good resolution graph can be obtained
from the minimal good embedded resolution graphs $\Gamma_i$ of the
plane curve singularities $(C,p_i)_{i=1}^\nu$ by adding one extra
vertex $v_C$ (corresponding to $C$), and  for all $i$ connecting
$v_C$ by one edge to the $(-1)$-curve of $\Gamma_i$. $\pi$ will
stay for this minimal good resolution and $\cU=\{v_C\}$.

All the needed invariants regarding the next discussion (excepting
$P^\cU$) were computed in \cite{SI}, \cite{MR2192386} and
\cite{ratopen}. We invite the reader to consult these articles for
more details. We will consider the $h=0$ case only.

We write $\Delta(t)$ for the product of the characteristic
polynomials (associated with the  monodromy) of all local plane
curve singularities $(C,p_i)\subset (\setC^2,p_i)$. One may verify
that $H=\setZ_d$  and 
$$Z_0^\cU(t)=\frac{1}{d}\sum_{\xi^d=1}\frac{\Delta(\xi t^{1/d})}{(1-\xi
t^{1/d})^2}.$$ On the other hand, the Poincar\'e series
(associated with the divisorial filtration of the projectivised
tangent cone) is given by (\ref{wh4}), namely
$$P_0^\cU(t)=\frac{1-t^d}{(1-t)^3}.$$
Following \cite{MR2192386}, we define
\begin{equation}\label{N}
N(t)\coloneqq \frac{1}{d}\sum_{\xi^d=1}\frac{\Delta(\xi
t^{1/d})}{(1-\xi t^{1/d})^2}-\frac{1-t^d}{(1-t)^3}.
\end{equation}
Clearly, the Reduced Identity (for $\cU=\{C\}$ and $h=0$) is
equivalent with the vanishing of $N(t)$. In
\cite[(2.4)]{MR2192386} is proved that $N(t)$ is a symmetric
polynomial (i.e. $N(t)=t^{d-3}\cdot N(1/t)$) with integral
coefficients and with $N(0)=0$. Moreover,
$$N(1)=sw(\Sigma)-(K_{\widetilde{X}}^2+s)/8-p_g(X,o),$$
where $sw(\Sigma)$ is the Seiberg--Witten invariant of $\Sigma$
associated with the canonical $spin^c$-structure,
$K_{\widetilde{X}}$ is the canonical class of $\widetilde{X}$,
$s=|\cV|$ as above, and $p_g(X,o)$ is the geometric genus of
$(X,o)$. Notice that the vanishing of $N(1)$ is exactly the
subject of the `Seiberg--Witten invariant conjecture', formulated
by L. Nicolaescu and the author, cf.
\cite{nemethi02:_seiber_witten,nemethi04:_seiber_witten}.

In particular, all superisolated hypersurface singularities which
are counter-examples for the Seiberg--Witten invariant conjecture
provide examples when the Reduces Identity (hence the Main
Identity too) is not valid. This cannot happen if $d\leq 4$, but
may appear starting from $d\geq 5$. The complete list of cases
when $N(t)$ is non-zero for $5\leq d\leq 6$, and more examples
with $d=7$ and even with arbitrary high $d$ are provided in
\cite[(2.7)-(2.8)]{MR2192386}. In all known cases (by the author)
$N(t)\not=0$ may occur only if the number $\nu$ of singular points
of $C$ is $\geq 2$.

In fact, in \cite{MR2192386} it is conjectured that $N(t)$ {\em
has only non-positive coefficients}. This conjecturally provides
a very strong criterion for the existence of cuspidal rational
planes curves with singularities with prescribed topological
types.

For unicuspidal curves (i.e. when $\nu=1$), in \cite{MR2192386} is
proved that $N(t)$ {\em has non-negative coefficients}. In
particular, the above conjecture about $N(t)$ transforms into the
conjecture which predicts the vanishing of $N(t)$. In fact, this
is also equivalent with the Seiberg-Witten invariant conjecture
for $(X,o)$. Its validity  was verified for all situation when the
logarithmic Kodaira dimension $\bar{\kappa}$ of
$\setC\setP^2\setminus C$ is $< 2$, and for $\bar{\kappa}=2$ in
all the examples known (by the authors of \cite{MR2192386}).

\begin{example}\label{graph}
Here is the case $C_4$ from \cite{SI} with $d=5$, where $C$ has
two singularities with multiplicity sequence $[3]$ and $[2_3]$.
The dual resolution graph of $f$ is presented below. Let
$\cU=\{C\}$, where $C$ is the $(-31)$-curve. In this case
$N(t)=-2t$, hence the Reduced Identity is not satisfied.

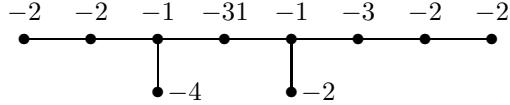
\begin{picture}(300,45)(20,0)
\put(125,25){\circle*{4}} \put(150,25){\circle*{4}}
\put(175,25){\circle*{4}} \put(200,25){\circle*{4}}
\put(225,25){\circle*{4}} \put(150,5){\circle*{4}}
\put(200,5){\circle*{4}} \put(100,25){\line(1,0){175}}
\put(150,25){\line(0,-1){20}} \put(200,25){\line(0,-1){20}}
\put(125,35){\makebox(0,0){$-2$}}
\put(150,35){\makebox(0,0){$-1$}}
\put(175,35){\makebox(0,0){$-31$}}
\put(200,35){\makebox(0,0){$-1$}}
\put(225,35){\makebox(0,0){$-3$}} \put(160,5){\makebox(0,0){$-4$}}
\put(210,5){\makebox(0,0){$-2$}} \put(100,25){\circle*{4}}
\put(250,25){\circle*{4}} \put(275,25){\circle*{4}}
\put(100,35){\makebox(0,0){$-2$}}
\put(250,35){\makebox(0,0){$-2$}}
\put(275,35){\makebox(0,0){$-2$}}
\end{picture}

%
\end{example}

\subsection{Example (a Newton non-degenerate
singularity)}\label{nn} Let $(X,o)$ be the zero set of the germ
$f=z_1^3+z_2(z_3^2+z_1^2z_2+z_2^6)$ (or any germ with the same
Newton diagram; in fact $f$ is a member of the bimodal singularity
$Q_{2,1}$). Let $\pi$ be its minimal resolution with dual graph:

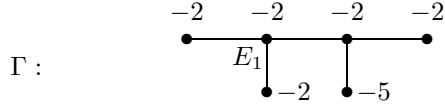
\begin{picture}(200,50)(100,10)
\put(200,40){\circle*{4}}
\put(230,40){\circle*{4}}\put(230,20){\circle*{4}}
\put(260,40){\circle*{4}} \put(260,20){\circle*{4}}
\put(200,40){\line(1,0){90}}
\put(260,40){\line(0,-1){20}}\put(230,40){\line(0,-1){20}}
\put(200,50){\makebox(0,0){$-2$}}
\put(260,50){\makebox(0,0){$-2$}}
\put(290,50){\makebox(0,0){$-2$}}
\put(240,20){\makebox(0,0){$-2$}}\put(270,20){\makebox(0,0){$-5$}}
\put(223,33){\makebox(0,0){$E_1$}}
\put(140,30){\makebox(0,0){$\Gamma:$}}
 \put(230,50){\makebox(0,0){$-2$}}
\put(290,40){\circle*{4}}
\end{picture}

$\Gamma$ can be obtained, e.g., by a toric resolution described in 
\cite{MR894303}. By this, there is a 1-1 correspondence between
the faces of the Newton diagram and the nodes of $\Gamma$ (see
also \cite{BN2}). By this correspondence, the exceptional divisor
(node) $E_1$ corresponds to the face $\Delta$ determined by the
monomials of $z_2(z_3^2+z_1^2z_2+z_2^6)$. The positive primitive
normal vector of this face is $(5,2,6)$, whose entries agree with
the vanishing orders of the three coordinates on $E_1$.

The group $H$ is $\setZ_{12}$, if $\rho$ is a generator of
$\widehat{H}$ then the possible values $\rho([E^*_i])$ are given
as follows (where $\xi$ is a primitive 12--root of unity):

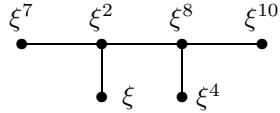
\begin{picture}(200,50)(100,10)
\put(200,40){\circle*{4}}
\put(230,40){\circle*{4}}\put(230,20){\circle*{4}}
\put(260,40){\circle*{4}} \put(260,20){\circle*{4}}
\put(200,40){\line(1,0){90}}
\put(260,40){\line(0,-1){20}}\put(230,40){\line(0,-1){20}}
\put(200,50){\makebox(0,0){$\xi^7$}}
\put(260,50){\makebox(0,0){$\xi^8$}}
\put(290,50){\makebox(0,0){$\xi^{10}$}}
\put(240,20){\makebox(0,0){$\xi$}}\put(270,20){\makebox(0,0){$\xi^4$}}
 \put(230,50){\makebox(0,0){$\xi^2$}}
\put(290,40){\circle*{4}}
\end{picture}

Our goal is to compare three series, namely $Z_0^\cU(t)$,
$P_0^\cU(t)$ (for $h=0$ and  $\cU=\{E_1\}$),  and $P^{\cG}(t)$
associated with the filtration $\cG$ constructed in (\ref{wh1}),
associated with weights ${\bf w}=(5,2,6)$, corresponding to the
face $\Delta$.

First, notice that $\Gamma$ is a minimally elliptic graph, and the 
resolution  is
minimal, hence $P(\bt)=Z(\bt)$ by (\ref{MTH}). In particular,
$P^\cU=Z^\cU$ too. $Z_0^\cU$ can be read from the graph, namely it
is the following Fourier--Dedekind sum (over the 12--roots $\xi$
of unity):
$$Z_0^\cU(t^{12})=\frac{1}{12}\sum_{\xi}\
\frac{(1-\xi^2t^{52})(1-\xi^8t^{40})}{(1-\xi^7t^{26})(1-\xi
t^{26})(1-\xi^{10}t^{20})(1-\xi^4t^8)}.$$ Hence, by a computation,
we get:
$$P_0^\cU(t)=Z_0^\cU(t)=\frac{(1-t^{15})(1-t^{12})}{(1-t^{13})(1-t^5)(1-t^2)(1-t^6)}.$$
Finally, notice that the series $P^\cG$ provided by the face
$\Delta: \ 5z_1+2z_2+6z_3=14$ is
$$P^\cG(t)=\frac{1-t^{14}}{(1-t^5)(1-t^2)(1-t^6)}.$$
This also shows that the divisorial filtration $\cF(l)$ (of $E_1$)
is strictly larger than the filtration $\cG(l)$ (for some $l$),
see (\ref{wh1}) for the terminology. Let us analyze the case
$l=12$ and 13 more closely. In the case of $\cG$ the weight 14 of
the equation is higher, so $Gr^\cG(l)$ is the same as the
corresponding vector spaces computed for the polynomial ring
$\cO$. We get a basis $\{z_3^2, z_3z_2^3,z_2^6,z_1^2z_2\}$ for
$Gr^\cG(12)$, and $\{z_1z_2z_3, z_1z_2^4\}$ for $Gr^\cG(13)$. On
the other hand, for $\cF$ we notice that along $E_1$, the
expression $h:=z_3^2+z_1^2z_2+z_2^6$ equals $z_1^3/z_2$ which has
order of vanishing 13. Hence $h$ provides a relation in
$Gr^\cF(12)=Gr^\cG(12)/(h=0)$, while $h$ is a new element in
$Gr^\cF(13)= Gr^\cG(13)\oplus \setC\langle h\rangle$. This fits
with the corresponding coefficients of the Poincar\'e series.

In particular, the `naive' series $P^\cG$ associated with a face
$\Delta $ of a Newton diagram does not agree (in general) with the
Poincar\'e series $P_0^\cU$ provided by the valuation of the
exceptional divisor corresponding to $\Delta$ (although they agree
for weighted homogeneous singularities). Moreover, although both
$P^\cG$ and $Z^\cU_0$ are topological/combinatorial, and both are
very natural, and one might hope that they should coincide, in
general, this is not the case. In general, $Z^\cU_0$ keeps more
information from the Newton boundary (from the complement of
$\Delta$ too). The author knows no direct formula of $Z_0^\cU$
directly from the combinatorics of the Newton diagram.

Finally notice that in this example the `leading term'
$z_2(z_3^2+z_1^2z_2+z_2^6)$ is not irreducible, causing the
anomaly $\cG\not=\cF^\cU$. This shows that the irreducibility
restriction in (\ref{wh4}) is necessary indeed.

\subsection{} In the next example we provide a general method
to find singularities with  $Z\not=P$. In fact, as this
construction suggests, the Main (Reduced) Identity may hold only
for very special, rigid analytic structures, and only for those
resolutions which `fit with some resolution  properties of the
analytic type' (see also (\ref{sq}) and (\ref{ex.7.new})).

\subsection{Example (for $Z\not=P)$}\label{me2}
Consider the following resolution graph $\Gamma$

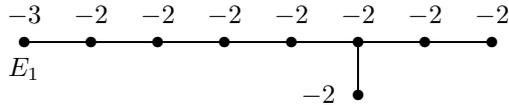
\begin{picture}(300,45)(20,0)
\put(125,25){\circle*{4}} \put(150,25){\circle*{4}}
\put(175,25){\circle*{4}} \put(200,25){\circle*{4}}
\put(225,25){\circle*{4}} \put(225,5){\circle*{4}}
\put(100,25){\line(1,0){175}} \put(225,25){\line(0,-1){20}}
\put(125,35){\makebox(0,0){$-2$}}
\put(150,35){\makebox(0,0){$-2$}}
\put(175,35){\makebox(0,0){$-2$}}
\put(200,35){\makebox(0,0){$-2$}}
\put(225,35){\makebox(0,0){$-2$}} \put(210,5){\makebox(0,0){$-2$}}
\put(100,25){\circle*{4}} \put(250,25){\circle*{4}}
\put(275,25){\circle*{4}} \put(100,35){\makebox(0,0){$-3$}}
\put(250,35){\makebox(0,0){$-2$}}
\put(275,35){\makebox(0,0){$-2$}}\put(100,15){\makebox(0,0){$E_1$}}
\end{picture}

It is realized e.g. by  $\{z_1^2+z_2^3+z_3^{11}=0\}$, but one
might take any analytic structure supported by $\Gamma$. Notice
that $H=1$.

First we consider the (minimal) resolution $\pi$ with the above
graph. Since $\Gamma$ is minimally elliptic, $Z=P$ by (\ref{MTH}).
Next, we concentrate on the reduced case $\cU=\{E_1\}$. The
identity $Z^\cU=P^\cU$ can be verified independently of
(\ref{MTH}) as well: $Z^\cU_\pi$ equals with the right hand side
of (\ref{ZU}) by a computation; while $P^\cU_\pi$ follows from
(\ref{wh4}) with weights $(3,2,1)$. In fact, the point what we
wish to make, is not about the resolution $\pi$.

Let $\Gamma_b$ be a non-minimal graph obtained from $\Gamma$ by
blowing up a regular point of $E_1$, and let $E_2$ be the new
exceptional divisor. From topological point of view, all the
possible choices of the center of this blow up are equivalent. But
this is not the case from analytic point of view. In order to
explain this we need a definition: For any resolution $\phi$,
$Z_{max}$ denotes  the divisorial part of $\phi^*(m_{X,o})$, where
$m_{X,o}$ is the maximal ideal of $\cO_{X,0}$.

Let us analyze closer the resolution $\pi$. Since the singularity
is minimally elliptic one has $Z_{max}=Z_{min}$. But, since
$Z_{min}^2=-1$ and the multiplicity is 2, $\pi^*(m_{X,o})$ has a
unique base point $Q$; in fact,
$\pi^*(m_{X,o})=\cO(-Z_{min})\cdot m_Q$, where $m_Q$ is the
maximal ideal of $Q$ (see \cite{Laufer77}). Since $E_1$ is the
only component with $(Z_{min},E_1)$ negative, we get that $Q$ is
on the regular part of $E_1$.

Let $\pi^{gen}$ (respectively $\pi^Q$) be the resolution obtained
from $\pi$ by blowing up a generic point ($\not=Q$) of $E_1$
(respectively $Q$).  In the first case $Z_{max}=Z_{min}$, while in
the second case $Z_{max}=Z_{min}+E_2$. Notice that the
multiplicity of $E_2$ in $Z_{min}$ is 1.

Therefore, considering the reduced situation for the graph
$\Gamma_b$ and $\cU=\{E_2\}$, we get the following: the
coefficient of $t$ in the series $P^{E_2}_{\pi^{gen}}(t)$  is
\emph{non-zero}, while in   $P^{E_2}_{\pi^Q}(t)$ is 0. Since
(computed in $\Gamma_b$)
$$Z^{E_2}(t)=\frac{t^6-1}{(t^3-1)(t^2-1)^2},$$
we get that $Z^{E_2}_{\pi^{gen}}\not=P^{E_2}_{\pi^{gen}}$. In
fact, one can show that the divisorial filtration associated with
$E_1$ in $\pi$ agrees with the divisorial filtration associated
with $E_2$ in $\pi^{gen}$, hence $P^{E_2}_{\pi^{gen}}=P^{E_1}_\pi$
(which agrees with the expression from (\ref{ZU})).


\begin{example}\label{ex.7.new}
Let us analyse the above examples from the point of view of
splice-quotient
singularities. In the resolution $\pi$ the minimally elliptic
structure is compatible with the splice-quotient analytic structure 
constructed from the graph of $\pi$ (they are equisingular). On the
other hand, on $\pi^{gen}$, there are two analytic structures 
(both minimally elliptic), which are `not compatible', they are
`different'.
The first is the pull-back of the splice-quotient analytic structure 
constructed by the graph of $\pi$, the other is the splice-quotient
structure constructed in $\pi^{gen}$. The point is that the first is
not even splice-quotient in $\pi^{gen}$ (e.g., it does not admit all
the `end curves' by its very construction). 
The splice-quotient analytic structure of $\pi^{gen}$ is the same
with a pulled back splice-quotient if and only if the blow up happened
at the base point, i.e. in a very special point. 

Hence, blowing up in a
{\em wrong point} a splice-quotient singularity we can get a
non-splice-quotient one. 

The same is true for the property $Z=P$.
\end{example}

\begin{example} The interested reader may construct (or find in
the literature) easily pairs of analytic structures of
singularities supported by the same topological type, and fixed
resolution graphs (even good minimal ones) such that the two
cycles $Z_{max}$ are nor equal. By the same argument as above, we
get that the Main Identity cannot hold for both analytic
structures.

Notice that analyzing the maximal ideal cycle as above we target
the very first coefficient of $Z(t)$. Evidently, there are many
more subtle analytic information (e.g. basepoints of other linear
systems)  which may obstruct the identity $Z=P$.
\end{example}


\section{The Relative Identities}\label{rel}

\subsection{The set-up.}\label{setuprel}
Let $(X,o)$ be a normal surface singularity and $(C,o)\subset
(X,o)$ be a reduced curve-germ on it. We fix a good embedded
resolution $\pi:\widetilde{X}\to X$ of $C\subset X$, which
satisfies the following restriction:
\begin{equation}\label{restr}\left\{
\begin{array}{l} \mbox{$\pi$ has no irreducible exceptional divisor
intersected by}\\ \mbox{two different components of the strict
transform of $C$.}\end{array}\right.
\end{equation}
 $L$ and $L'$
denote the corresponding lattices associated with $\pi$.
Furthermore, we consider the universal abelian cover $c:(Y,o)\to
(X,o)$ and the maps $\widetilde{c}$ and $\pi_Y$ as in
(\ref{setup}). For any $j\in \cV$, $\{k(j)\}$ stays for the index
set of the irreducible exceptional divisors of $\pi_Y$ which are
projected by $\widetilde{c}$ onto $E_j$.

Let $\widetilde{C}$ be the $\pi$-strict transform of $C$,
$\cU^C\subset \cV$ be the set of those irreducible exceptional
divisors which intersect $\widetilde{C}$, and $\widetilde{C}_j$
($j\in\cU^C$) the component of $\widetilde{C}$ which intersects
$E_j$. The irreducible components of
$\widetilde{c}^{-1}(\widetilde{C}_j)$ are denoted by
$\{\bar{C}_{j,\alpha}\}_\alpha$.

\begin{definition}\label{filrel}
We define the following filtration of $\cO_{Y,o}$, associated with
$\pi$ and the pair $(C,o)\subset (X,o)$, and indexed by $l'\in
L'$:
\begin{equation}\label{FIL}
\cF^C(l'):=\Big\{f\in \cO_{Y,o}\ : \
\begin{array}{ll}
div(f\circ \pi_Y)_{k(j)}\geq \widetilde{c}^*(l')_{k(j)} &
\mbox{ for any $k(j)$ with $j\not\in \cU^C$}\\
(div(f\circ \pi_Y)- \widetilde{c}^*(l'),\bar{C}_{j,\alpha})\geq 0
& \mbox{ for any $\alpha$ with $j\in \cU^C$}
\end{array}\Big\}.
\end{equation}
Above, in the second line, the intersection multiplicity
$(\cdot,\cdot)$ takes values in $\setN\cup\infty$,
$(D,\bar{C}_{j,\alpha})$ being $\infty$ if $D$ contains
$\bar{C}_{j,\alpha}$ as a component.
\end{definition}
\noindent (Above, for $j\in \cU^C$, even if along the exceptional
divisor the divisor  $div(f\circ \pi_Y)$ is not larger than $\widetilde{c}^*(l')$,
its intersection multiplicity with $\bar{C}_{j,\alpha}$ can be
larger if one of its  non-compact components has a  large 
intersection multiplicity.) 

Similarly as in the non-relative version, one denotes by
$\hh^C_h(l')$, ($[l']=h$) the dimension of the
$\theta(h)$-eigenspace of $\cO_{Y, o}/\cF^C(l')$, and one defines
$H^C$, $L^C$ and $P^C$ by similar formulas as in the non-relative
case. E.g.:
\begin{equation}\label{eq:4C}
P^C(\bt)=-H^C(\bt)\cdot \prod_j(1-t_j^{-1}) \ \ \mbox{in
$\setZ[[\bt^{L'_e}]]$}.
\end{equation}

In the combinatorial setting one does the following modification.
Define $\delta^C_j$ as $\delta_j$ for any $j\not\in\cU^C$,
otherwise set $\delta^C_j:=\delta_j+1$. Then one defines $\zeta^C$
and $Z^C$ similarly as the corresponding non-relative invariants,
but using $\delta^C_j$ instead of $\delta_j$. By the same argument
as in (\ref{TH1}), the identity  $\zeta^C=Z^C$  follows.

\begin{definition}\label{def:rel}
We say that the pair $(C,o)\subset (X,o)$ and the resolution $\pi$ satisfies the
{\bf Relative Identities} if
$$Z^C(\bt)=P^C(\bt).$$
\end{definition}

\begin{theorem}\label{TH:rel}
Assume that $(X,o)$ and \underline{all}  its good resolutions
satisfy the `Main Identities' (\ref{MI}). Then, for any curve-germ
$(C,o)\subset (X,o)$ and good resolution $\pi$ as in
(\ref{setuprel}), the `Relative Identities' are satisfied too.
\end{theorem}

\begin{proof} By the definitions, $\cF(l')\subset \cF^C(l')$ for any $l'\in L'$.
Moreover, if $f\in \cF^C(l')$, then either $f\in\cF(l')$, or the $\pi_Y$-strict transform of
$\{f=0\}$ intersects $\widetilde{c}^{-1}(\widetilde{C})$.

We fix an integer $n>0$, and we will focus on those coefficients
of $P^C(\bt)=\sum_{l'}\pp^C(l')\bt^{l'}$ for witch $l'_j\leq n$ for
any $j\in\cV$. For two series $A(\bt)=\sum a(l')\bt^{l'}$ and
$B(\bt)=\sum b(l')\bt^{l'}\in \setZ[[\bt^{L'_e}]]$ we write
$A\equiv^n B$ if $a(l')=b(l')$ for all $l'\in L'$ with $l'_j\leq
n$ for all $j$. Our goal is to show that  $P^C\equiv^n Z^C$ for
any arbitrarily chosen $n$.

Let $\bar{\pi}$ be the resolution obtained from $\pi$ by blowing
up the intersection of any strict transform component
$\widetilde{C}_j$ $(j\in \cU^C)$, with the exceptional divisor in
$m>0$ infinitezimally closed points. In this way we introduce, for
each $j\in\cU^C$, a string of new irreducible exceptional divisors
$E_{j1},\ldots, E_{jm}$, where all of them are $(-2)$ curves
excepting the last one $E_{jm}$ which is a $(-1)$ curve, $E_{j1}$
intersects $E_j$ and $\delta_{jm}=1$. Let $\bar{\cV}$ be the set
of irreducible exceptional divisors of $\bar{\pi}$, let $\cE$ be
the union of components of type $E_{jm},\ (j\in\cU^C)$, and
finally $\bar{\cU}:= (\cV\setminus \cU^C)\cup \cE\subset
\bar{\cV}$.

In the next paragraphs we will compare $\cF^C$ (associated with
$\pi$) with $\cF^{\bar{\cU}}$ (associated with $\bar{\pi}$). There
is a natural identification of their $t$-variables: the variables
corresponding to $\cV\setminus \cU^C=\bar{\cU}\setminus  \cE$ are
identified,  and also for any $j\in\cU^C$ the variable
corresponding to $E_j$ with the variable which corresponds to
$E_{jm}$. According to this correspondences, for any
$l'=\sum_jl'_jE_j\in L'$ we write
$$\bar{l}'=\sum_{j\in \cV\setminus \cU^C}l'_jE_j+\sum_{j\in \cU^C} l'_jE_{jm}\in \bar{L}'.$$
Notice also that $[E^*_{jm}]=[E^*_j]$ in $H$ for any $j\in \cU^C$.

Next, we formulate two  properties used later in the proof. (1)
uses the numerical behaviour of the intersection multiplicity with
respect to a blow up (of a smooth point), and the fact that the
strict transforms of $f$ and $C$ after sufficiently many blow ups
can be separated. (2) follows  from standard  arithmetical
properties of negative definite graph. Their detailed verification
is left to the reader. The statements are: for any fixed $n>0$,
there exists an integer $m(n)>0$ such that $\bar{\pi}$ constructed
with $m\geq m(n) $ satisfies the following two properties:
\begin{enumerate}
\item\label{egym}
 for any $l'\in L'$
with $l'_j\leq n$ ($j\in\cV$) one has $\cF^C_\pi(l')=\cF_{\bar{\pi}}^{\bar{\cU}}(\bar{l}')$;
\item\label{kettom}
$-(E^*_{jm},E^*_{jm})_{\bar{\pi}}>n$.
\end{enumerate}
For any $n$, fix such a $\bar{\pi}$. Then, by (\ref{egym}),  $P^C_\pi\equiv^n P^{\bar{\cU}}_{\bar{\pi}}$.
But $P^{\bar{\cU}}_{\bar{\pi}}= Z^{\bar{\cU}}_{\bar{\pi}}$ by the assumption that $(X,o)$ satisfies  the
main identities for $\bar{\pi}$, hence by (\ref{MR}) it satisfies the reduced identities too.
But $Z^{\bar{\cU}}_{\bar{\pi}}$ and $Z^C_\pi$ differ only by the contribution given by
$E_{jm}$ (in the first one), hence using (\ref{kettom}) and the fact that all these series are in
$\setZ[[\bt^{L'_e}]]$, we have
 $Z^{\bar{\cU}}_{\bar{\pi}}\equiv^n Z^C_\pi$.
In particular, $P^C\equiv^n Z^C$ for any $n$, hence they are equal.
\end{proof}

The next fact follows from (\ref{MTH}) and (\ref{TH:rel}). It
generalizes the main result of \cite{CDGc} (here we have a variable
even for $j\not \in\cU^C$; see also (\ref{RRR})).

\begin{corollary}\label{relrac}
Any curve $(C,o)$ on a rational surface singularity $(X,o)$, and
any resolution $\pi$ satisfy the Relative Identities.
\end{corollary}

\begin{remark}\label{ME}
The analog of (\ref{relrac}) for minimal elliptic singularities is
not true, see (\ref{EXMinEll}).
\end{remark}

\section{The Reduced Relative Identities}

\subsection{}\label{setupRRI}
Assume that we are in the situation of (\ref{setuprel}). By a
similar construction as in \S \ref{sec:RI} we can reduce the
number of variables, and we can formulate the `Reduced Relative
Identities' for any non-empty subset $\cU\subset\cV$. Moreover,
repeating the proof of (\ref{MR}) (combined with (\ref{TH:rel}))
we get

\begin{proposition}\label{PPP}
If the Main Identities are satisfied for $(X,o)$ and for all the
resolutions $\pi$, then the Reduced Relative Identities are
satisfied for any $(C,o)\subset (X,o)$ and $\cU\subset \cV$ as well.
\end{proposition}

In the sequel,  $\bt$ denotes the multivariable $\{t_j\}_{j\in
\cU}$.

In this section we consider that case when
$\cU$ is exactly the subset $\cU^C$ defined in (\ref{setuprel}),
namely, is the index set of irreducible exceptional divisors
intersected by the strict transform of $C$. As we will see, for
this special $\cU$, both invariants $Z^{C,\cU}$ and $P^{C,\cU}$
are {\em independent of the resolution}, and can be related  with
important classical invariants.

\subsection{The topological part.}\label{toppart} Let $K_C:=C\cap \Sigma$
be the link of $C$ in the link $ \Sigma$ of $X$. Set $G:=
H_1(\Sigma\setminus K_C,\setZ)$. Then $G/Tors(G)$ is isomorphic
with $\setZ^r$, where $r$ is the number of irreducible components
of $C$ ($=|\cU^C|$). Let $\tau:G\to \setZ^r$ be this natural
morphism. Moreover, for any character $\rho\in \widehat{H}$, let
$\rho^e\in \widehat{G}$ be its extension  via the natural map
$G\to H$ induced by the inclusion.

Then, the Milnor-Reidemeister torsion \cite{Milnor66} of
$\Sigma\setminus K_c$ associated with the representation
$(\tau,\rho^e)$ can be identified (up to the natural ambiguities
of the torsion) with
\begin{equation}\label{torsion}
\widehat{Z}_\rho(\bt)=\prod_{j\in
\cV}(1-\rho([E_j^*])\bt^{pr_\cU(E^*_j)})^{\delta_j^C-2},
\end{equation}
where this expression  is written in terms of a
resolution $\pi$ which satisfies (\ref{restr}). (Since this fact
will be not used later, we omit its proof.) As such, {\em
$\widehat{Z}_\rho(\bt)$ depends only on the embedding $K_C\subset
\Sigma$ and on $\rho$, and it is independent of the resolution
chosen} (a fact, which can be verified by a direct computation as
well).

E.g., when $\Sigma$ is an integral homology sphere, then $\rho=1$,
and $\widehat{Z}_1(\bt)$ is the Alexander polynomial of
$K_C\subset \Sigma$, see \cite{EN}.

Notice that for any $h\in H$ one has
$$Z^C_h(\bt)=\frac{1}{|H|}\ \sum_{\rho}\
\rho(h)^{-1}\, \widehat{Z}_\rho(\bt).$$ Therefore, $\{Z^C_h\}_h$
and $\{\widehat{Z}_\rho\}_\rho$ correspond by Fourier transform.

\begin{example} There is an important specialization of
$\widehat{Z}_\rho(\bt)$. Assume that $f:(X,o)\to(\setC,o)$ is a
non-zero analytic function-germ. For simplicity, we will assume
that it defines an isolated singularity. Set $C:=\{f=0\}$. Let
$\pi$ be an embedded good resolution of $C\subset X$ with
(\ref{restr}), and denote by $m_j$ the order of vanishing of
$f\circ \pi$ along $E_j$ ($j\in\cV$). Since
$(div(f\circ\pi),E^*_j)=0$, we get that
$m_j=-\sum_{i\in\cU}(E^*_i,E^*_j)$, or in other words $t^{m_j}$
can be obtained by specialization
\begin{equation}\label{sub}
t^{m_j}=\bt^{E^*_j}|_{t_i\mapsto t\,\mbox{\tiny{for
all}}\,i\in\cU}.
\end{equation}
Consider the Milnor fibration $\arg(f):\Sigma\setminus K_C\to S^1$
of $f$ with Milnor fiber $F$. Let $\setC_\rho$ be the flat-bundle
associated with representation $\rho$ above $\Sigma$. Then on can
consider the algebraic monodromy action induced by the Milnor
fibration on $H_*(F,\setC_\rho)$ (with twisted coefficients), let
$\zeta_\rho(t)$ be its zeta function. Then, one can prove
(similarly as A'Campo's formula for the usual monodromy
zeta-function, see \cite{A'C}) that
$$\zeta_\rho(t)=\widehat{Z}_\rho(\bt)|_{t_i\mapsto t\,\mbox{\tiny{for
all}}\,i\in\cU}.$$ In particular, if $C$ is irreducible, then
$\zeta_\rho(t)=\widehat{Z}_\rho(t)$.  If $\rho=1$, then
$\zeta_\rho$ is just the usual monodromy zeta function of $f$.
\end{example}

\subsection{The analytical part.} Consider the setup of
(\ref{setupRRI}). Set $C^Y:=c^{-1}(C)$. Then $c:C^Y\to C$ is a
branched covering (regular over $C\setminus o$) with Galois group
$H$.

First,  we show that {\em $P^{C,\cU}$ depends only on the covering
$c:C^Y\to C$}. Indeed,  for any $l'\in pr_\cU(L')$, (\ref{FIL})
reads as
\begin{equation}\cF^{C,\cU}(l'):=\big\{f\in
\cO_{Y,o}\ : \  (div(f\circ \pi_Y),\bar{C}_{j,\alpha})\geq l'_j\
\mbox{ for any $\alpha$ and $j\in \cU$}\big\}.
\end{equation}
Therefore, if $f\in\cI(C^Y)$, the ideal of $C^Y$ in $\cO_{Y,o}$,
then $  (div(f\circ \pi_Y),\bar{C}_{j,\alpha})=\infty$ for any
component $\bar{C}_{j,\alpha}$, hence $\cI(C^Y)\subset
\cF^{C,\cU}(pr_\cU(l'))$ for any $l'\in L'$.
 In particular, any of the series $H$, $L$ or $P$ associated with
 the pair $(C,\cU)$ depends only on the induced  filtration of
 $\cO_{C^Y}=\cO_{Y,o}/\cI(Y^C)$ and the action of $H$ on it.
On the other hand, the filtration also can be recovered
intrinsically from $C^Y$. Let $\{C_j\}_{j\in \cU}$ be the
irreducible decomposition of $C$, let $\{C_{j,\alpha}\}_\alpha$ be
the irreducible decomposition of $c^{-1}(C_j)$. Notice that the
restriction of $\pi_Y$ to $\bar{C}_{j,\alpha}$ is, in fact,  the
normalization $n_{j,\alpha}$ of $C_{j,\alpha}$. We fix a local
coordinate $s_{j,\alpha}$ in $\bar{C}_{j,\alpha}$, and we define
the valuation $v_{j,\alpha}$ on $\cO_{C^Y}$ by $v_{j,\alpha}(f)$
being the multiplicity of $n_{j,\alpha}\circ f$, as a power series
in $s_{j,\alpha}$. Then, for any $l'\in pr_\cU(L')$, one has
\begin{equation}\cF^{C,\cU}(l'):=\big\{f\in
\cO_{C^Y}\ : \  v_{j,\alpha}(f) \geq l'_j\ \mbox{ for any $\alpha$
and $j\in \cU$}\big\}.
\end{equation}
For any $l'\in pr_\cU(L')$ and  $h\in H$,  let $\hh^{C,\cU}_h(l')$
be the dimension of the $\theta(h)$-eigenspace of
$(\cO_{C^Y}/\cF^{C,\cU}(l'))_{\theta(h)}$. Then
$$H^{C,\cU}_h(\bt)=\sum_{l'\in pr_\cU(L')}\hh^{C,\cU}_h(l')\cdot
\bt ^{l'}.$$ One can describe $L^{C,\cU}_h$ and $P^{C,\cU}_h$
similarly.

If $h=0$, then the situation is more simpler: $H^{C,\cU}_0$,
$L^{C,\cU}_0$ and $P^{C,\cU}_0$ depend only on the curve $C$.
Indeed, let $v_j$ be the valuation of $\cO_C$ defined via the
normalization $n_j:\bar{C}_j\to C_j$. Then, for any $l\in
L_\cU:=pr_\cU(L)$ set
$$\cF(l):=\big\{f\in
\cO_{C}\ : \  v_{j}(f) \geq l_j\ \mbox{ for any $j\in
\cU$}\big\}.$$ Then
$$H^{C,\cU}_0(\bt)=\sum _{l\in L_\cU} \ \dim(\cO_C/\cF(l))\cdot
\bt^l.$$ 
\begin{example}\label{A3}
The covering $c:C^Y\to C$, in general, is not trivial. Moreover,
$c$ cannot be recovered form the curve $C$ and the abstract group
$H$ only:  the covering $c:C^Y\to C$  preserves some information
from the universal abelian cover $c:Y\to X$. E.g., let $(X,o)$ be
the hypersurface singularity given by $\{x^2+y^2-z^4=0\}$, and $C$
be the smooth curve on $X$ given by $\{x=y-z^2=0\}$. (If $\pi$ is
the minimal resolution of $X$, then $\pi^{-1}(C)$ is a normal
crossing divisor, and the strict transform of $C$ intersects the
component with $\delta_j=2$.) Then $c:C^Y\to C$ corresponds to the
representation $\pi_1(C\setminus o)=\setZ\to \setZ_4$ given by
$1\mapsto \hat{2}$.
\end{example}

\subsection{}\label{statement} {\bf The identity} $Z^{C,\cU}(\bt)=P^{C,\cU}(\bt)$ (or any
of its reduced identities) {\em is hardly non-trivial, and, in
fact, is rather special and  surprising if it holds}.

Nevertheless,  (\ref{MTH}) and (\ref{PPP}) provide the next result
for rational singularities (which is the main result of
\cite{CDGc}):

\begin{corollary}\label{RRR} $Z^{C,\cU}_h(\bt)=P^{C,\cU}_h(\bt)$ holds
for any $h\in H$ and  for any curve
$C$ embedded into a rational singularity $(X,o)$.
\end{corollary}

Let us return back to the first sentence of (\ref{statement}), and
analyze it more explicitly for $h=0$. By the above discussion,
$$Z^{C,\cU}_0(\bt)=\frac{1}{|H|}\sum_{\rho}\widehat{Z}_\rho(\bt)$$
is a subtle topological invariant of the embedding $C\subset X$
(or $K_C\subset \Sigma$), while $P_0^{C,\cU}(\bt)$ depends only on
(analytic type) of the (abstract) curve $C$.

The validity of the identity $Z_0^{C,\cU}=P_0^{C,\cU}$ (say, for
some family of normal surface singularities $X$) can be translated
as follows: {\em the analytic type of an  abstract curve $C$ imposes
that if $C$ is embedded into some $X$ (from that family), the
Alexander invariant of the embedding is independent of the
embedding}.

E.g., if a smooth irreducible curve $C$ can be embedded into $X$
then $Z_0^{C,\cU}=P_0^{C,\cU}$ implies that the Alexander
invariant of the embedding should be $\sum_{k\geq 0}t^k=1/(1-t)$
(the Alexander invariant of the trivial knot in $S^3$, or the
Poicar\'e polynomial of $\setC[s]$, where $deg(s)=1$).

\begin{example}
Consider the situation from (\ref{A3}) (when $X$ is rational).
Then (written in the minimal resolution)
$$P(t)=\frac{1}{4}\sum_{\xi^4=1}\frac{1-\xi^2t}{(1-\xi t^{1/2})(1-\xi^3t^{1/2})}.$$
After a computation this, indeed, equals  $1/(1-t)$.
\end{example}

\begin{example}\label{EXMinEll}
The statement of (\ref{RRR})  is not true if we consider arbitrary
embeddings into minimally elliptic singularities, even if the
curve $C$ is smooth irreducible (unless $C$ is embedded in a very
special way).

Consider the minimally elliptic singularity $X=\{x^2+y^3+z^7=0\}$.
Its minimal good resolution graph is $\Gamma$ from (\ref{237})
whose notations we will use. Let $C$ be the irreducible curve
$\{z=x^2+y^3=0\}$. The minimal good resolution resolves the pair
$C\subset X$ as well, the strict transform cuts $E_1$. Hence,
$Z(t)$, which in this case equals also the monodromy zeta function
of $z:X\to \setC$ is $(t^6-1)/(t^3-1)(t^2-1)=(t^2-t+1)/(1-t)$,
hence equals the monodromy zeta function of the plane curve
singularity $x^2+y^3$ (i.e., of $C$ embedded differently), and
equals also the Poincar\'e series $P(t)$ (which in this case is
the same as the generating function of the semigroup
$\{0,2,3,4\cdots \}\subset \setN $ of $C$). Hence
$Z^{C,\cU}=P^{C,\cU}$.

Consider now another irreducible curve $C'$ in $X$ which has the
same embedded topological type as $C$, i.e. its strict transform cuts
transversally $E_1$. E.g., one of the possibilities is to take
$C':=\{y=z^2, x=iz^3\sqrt{1+z}\}$. (Notice that the
parametrization $z\mapsto (iz^3\sqrt{1+z}, z^2,z)$ has leading
terms of weights $(3,2,1)$, which fits with the weights considered
for $E_1$ in (\ref{237}).) Clearly, $Z^{C',\cU}$ is as before,
hence equals $(t^2-t+1)/(1-t)$. But, since $C'$ is smooth,
$P^{C',\cU}(t)=1/(1-t)$. Hence $Z^{C',\cU}\not=P^{C',\cU}$. This
also shows that the non-reduced, relative identity also fails in
general (cf. \ref{ME}).

This construction is not special only for this case. (E.g., by a
very small modification, one may repeat the same argument for
(\ref{me2}) too.) Definitely, the main difference between $C$ and
$C'$ is that one of them contains the basepoint $Q$ of
$\pi^*(m_{X,o})$ the other one not. (Compare with the discussion
from (\ref{me2}).) Hence, possibilities to construct such pairs
appear in abundance.
\end{example}

\bibliographystyle{amsplain}\bibliography{hivatkozas}

\end{document}